\documentclass[10pt]{article}

\usepackage[dvipsnames]{xcolor}
\usepackage{leftidx}
\usepackage{amsmath}
\usepackage{amscd}
\usepackage{amssymb}
\usepackage{rotating}
\usepackage{amsfonts}
\usepackage{amsthm}
\usepackage{verbatim}
\usepackage{stmaryrd}
\usepackage{mathrsfs}
\usepackage{color}
\usepackage[all]{xy}

\usepackage[top=1in, bottom=1in, left=1.25in, right=1.25in]{geometry}

\usepackage{titlesec}
       \titleformat{\chapter}[display]
             {\normalfont\Large\bfseries}{\thechapter}{11pt}{\Large}
       \titleformat{\section}
             {\normalfont\large\bfseries}{\thesection}{11pt}{\large}
       \titlespacing*{\chapter}{0pt}{0pt}{15pt} 
       \titlespacing*{\section}{0pt}{3.5ex plus 1ex minus .2ex}{2.3ex plus .2ex}

    \usepackage[titletoc]{appendix}

\usepackage[colorlinks,linkcolor=black, anchorcolor=blue, urlcolor=blue, citecolor=blue]{hyperref}
\usepackage{cleveref} 
 \usepackage[hyperpageref]{backref}

\allowdisplaybreaks[1]

\DeclareMathAlphabet{\mathscrbf}{OMS}{mdugm}{b}{n}

\newcommand{\pqed}{\hfill\qedsymbol\\}

\newcommand{\Spec}{\mathrm{Spec}}

\newcommand{\st}{st}

\newtheorem{theorem}{Theorem}[section]
\newtheorem{theorem/definition}{Theorem/Definition}[section]

\newtheorem{proposition}[theorem]{Proposition}
\newtheorem{lemma}[theorem]{Lemma}

\newtheorem{corollary}[theorem]{Corollary}

\newtheorem{conjecture}[theorem]{Conjecture}

\theoremstyle{remark}
\newtheorem{remark}[theorem]{Remark}

\theoremstyle{definition}
 \newtheorem{example}[theorem]{Example}
\newtheorem{definition}[theorem]{Definition}
\newtheorem{construction}[theorem]{Construction}

\begin{document}
\title
{\large{\textbf{An inclusion-exclusion principle for tautological sheaves on Hilbert schemes of points}}}
\author{\normalsize Xiaowen Hu}
\date{}
\maketitle

\begin{abstract}
We show an equation of Euler characteristics of tautological sheaves on Hilbert schemes of points on the fibers of a double point degeneration. This equation resembles a computation of such Euler characteristics via a combinatorial inclusion-exclusion principle. As a consequence, we show the existence of universal polynomials for the Euler characteristics of tautological sheaves on the Hilbert scheme of points on smooth proper algebraic spaces.  We apply this result to a conjecture of Zhou on tautological sheaves on Hilbert schemes of points, and reduce the conjecture to the cases of products of projective spaces. Our main tools are good degenerations and algebraic cobordism.
\end{abstract}





\tableofcontents

\section{Introduction}
Let $C$ be a smooth curve over a base field $\Bbbk$, $\mathfrak{X}$ be a smooth scheme (or more generally, a smooth algebraic space) over $\Bbbk$ and $\pi:\mathfrak{X}\rightarrow C$ a proper and flat morphism. Let $0$ be a $\Bbbk$-point of $C$. We say that $\pi$ is a \emph{double point degeneration} if $\pi$ is smooth away from $0$, and $\pi^{-1}(0)$ is a normal crossing divisor which can be written as $Y_1\cup Y_2$ such that $Y_1$ and $Y_2$ are smooth divisors and  $D=Y_1\cap Y_2$ is smooth. Let $\xi$ be a $\Bbbk$-point of $C\setminus\{0\}$, and $X$ be the smooth fiber $\mathfrak{X}_{\xi}$. By abuse of notations, we say that $X\overset{\mathfrak{X}}{\rightsquigarrow}Y_1\cup_D Y_2$ is a double point degeneration when we want to emphasize the fibers, rather than the total space.

For a $\Bbbk$-scheme $Y$, we are going to study the Euler characteristics of  tautological sheaves on $\mathrm{Hilb}^n(Y)$, the Hilbert scheme of $n$-points on $Y$. By definition there is a universal closed subscheme $\mathcal{Z}$ of $\mathrm{Hilb}^n(Y)\times Y$ and a diagram
\[
\xymatrix{
	\mathcal{Z} \ar[r]^{q} \ar[d]^{p} & Y \\
	\mathrm{Hilb}^n(Y) & .
}
\]
For a vector bundle $E$ on a $\Bbbk$-scheme $Y$, we define $E^{[n]}=p_*q^* E$. It is a vector bundle on $\mathrm{Hilb}^n(Y)$, and  we call it the \emph{tautological sheaf associated with} $E$. 

For a vector bundle $G$ on a scheme $Y$ and a formal variable $u$, we define a polynomial in $u$ with coefficients in $K(Y)$:
\[
\Lambda_{-u}G:=\sum_{i=0}^{\mathrm{rank}\ G}(-u)^i \wedge^iG
\]
Moreover, for two vector bundles $G_1$ and $G_2$, we define
\[
\chi(G_1,G_2):=\sum_{i=0}^{\infty}(-1)^i \dim_{\Bbbk} \mathrm{Ext}^{i}_{\mathcal{O}_Y}(G_1,G_2).
\]
When $Y$ is a proper scheme and $G_1$ is locally free, this is a finite sum and thus $\chi(G_1,G_2)$ is well-defined. 

In this paper we prove the following \emph{inclusion-exclusion principle}.
\begin{theorem}\label{thm-intro-inc-exc-principle}
Let $\Bbbk$ be a field of characteristic zero. Let $X$, $Y_1$, $Y_2$ and $D$ be smooth proper algebraic spaces over $\Bbbk$, and $C$ a smooth curve over $\Bbbk$. 
Let $X\overset{\mathfrak{X}}{\rightsquigarrow}Y_1\cup_D Y_2$ be a double point  degeneration with total space $\mathfrak{X}\rightarrow C$. Let $E$ and $F$ two vector bundles on $\mathfrak{X}$. 
Let $\mathbb{P}_D=\mathbb{P}(N_{Y_1/D}\oplus \mathcal{O}_D)$. Then 
\begin{eqnarray*}
&&\log\Big(1+\sum_{n=1}^{\infty}\chi\big(\Lambda_{-u}(E|_{X}^{[n]}),
\Lambda_{-v}(F|_{X}^{[n]})\big)Q^n\Big)\\\
&=&\log\Big(1+\sum_{n=1}^{\infty}\chi\big(\Lambda_{-u}(E|_{Y_1}^{[n]}),
\Lambda_{-v}(F|_{Y_1}^{[n]})\big)Q^n\Big)
+\log\Big(1+\sum_{n=1}^{\infty}\chi\big(\Lambda_{-u}(E|_{Y_2}^{[n]}),
\Lambda_{-v}(F|_{Y_2}^{[n]})\big)Q^n\Big)\\
&&-\log\Big(1+\sum_{n=1}^{\infty}\chi\big(\Lambda_{-u}(E|_{\mathbb{P}_D}^{[n]}),
\Lambda_{-v}(F|_{\mathbb{P}_D}^{[n]})\big)Q^n\Big),
\end{eqnarray*}
as formal series of $u,v$ and $Q$. Here $E|_{\mathbb{P}_D}$ (resp. $F|_{\mathbb{P}_D}$) is the pullback of $E$ (resp. $F$) via the composition $\mathbb{P}_D\twoheadrightarrow D\hookrightarrow \mathfrak{X}$.
\end{theorem}


Let us explain the title. An equality of series of $Q$
\begin{equation}\label{eq-inclusion-exclusion-0}
\log(1+\sum_{i=1}^{\infty}c_i Q^i)=\log(1+\sum_{i=1}^{\infty}a_i Q^i)+\log(1+\sum_{i=1}^{\infty}b_i Q^i)-\log(1+\sum_{i=1}^{\infty}d_i Q^i)
\end{equation}
regarded as a system of equations for $a_i,b_i,c_i,d_i$, $i\geq 1$, 
is equivalent to
\begin{eqnarray*}
c_1&=&a_1+b_1-d_1,\\
c_2&=&a_2+b_2+a_1 b_1-a_1d_1-b_1d_1-d_2+d_1^2,\\
c_3&=&a_3 + a_2 b_1 + a_1 b_2 + b_3 - a_2 d_1- a_1 d_2- d_3 - a_1 b_1 d_1- b_1 d_2 - b_2 d_1\\
&&+ a_1 d_1^2 + b_1 d_1^2   + 2 d_1 d_2- d_1^3,\\
c_4&=& a_4 + a_3 b_1 + a_2 b_2 + a_1 b_3 + b_4\\
&& - a_3 d_1 - a_2 d_2- a_1 d_3- d_4
- a_2 b_1 d_1   - a_1 b_1 d_2- b_1 d_3 - a_1 b_2 d_1 - b_2 d_2- b_3 d_1\\
 && + a_2 d_1^2 + 2 a_1 d_1 d_2+ 2 d_1 d_3 + d_2^2   + a_1 b_1 d_1^2+ 2 b_1 d_1 d_2 + b_2 d_1^2 \\
 &&- a_1 d_1^3 - 3 d_1^2 d_2- b_1 d_1^3 + d_1^4, 
\end{eqnarray*}
and so on. The intuition for the corresponding formulae of Euler characteristics is the following. For example, to compute $\chi(\Lambda_{-u}E^{[2]},\Lambda_{-v}F^{[2]})$ on the special fiber, it seems reasonable to sum the contributions from the space of two points on $Y_1$, and that from two points on $Y_2$, and that of the first point on $Y_1$ and the second on $Y_2$; then subtract the contributions arising from repeated counting, i.e. the first point on $Y_1$ and the second  on $D$, the first  on $Y_2$ and the second on $D$, two points on $D$; finally we need to add back the repeated subtracted contributions from the first and the second points both on $D$. 
This is reminiscent of the classical inclusion-exclusion principle in enumerations. 
More generally, for $l\geq 1$, let 
\[
S_l(n)=\{\big((i_1,j_1),\dots,(i_l,j_l)\big)\in (\mathbb{Z}_{\geq 0}^2)^l|
i_1>i_2>\dots>i_l,\ i_k+j_k=n\ \mbox{for}\ 1\leq k\leq l
\}.
\]
\begin{lemma}\label{lem-intro-inc-exc}
The equation (\ref{eq-inclusion-exclusion-0}) is equivalent to
\begin{equation}\label{eq-inclusion-exclusion-1}
c_n=\sum_{l=1}^{n+1}\big((-1)^{l-1}\sum_{S_l(n)}a_{i_l}d_{i_1-i_2}d_{i_2-i_3}\cdots d_{i_{l-1}-i_l}b_{j_1}\big),
\end{equation}
for all $n\geq 1$.
\end{lemma}
This formula suggests a decomposition of a certain specialization of the Hilbert scheme of points. In this paper we realize this idea to prove Theorem 
\ref{thm-intro-inc-exc-principle} using Li-Wu's good degeneration \cite{LW15}. 

By the theory of algebraic cobordism of vector bundles (\cite{LeeP12}), we deduce from Theorem \ref{thm-intro-inc-exc-principle} the existence of \emph{universal polynomials}. 
Let us introduce some notations. 
Let $\mathcal{C}_{n,r,s}$ be the vector space of degree $n$  polynomials in the ring  
\[
\mathbb{Q}[u_1,\dots,u_n,v_1,\dots,v_{r},w_1,\dots,w_{s}].
\]
Let $E$ and $F$ be vector bundles on $X$ of ranks $r$ and $s$ respectively. 
We evaluate each monomial 
 \begin{equation*}\label{eq-intro-basis-chernpolynomial}
 u_1^{k_1}\cdots u_n^{k_n}\cdot v_1^{l_1}\cdots v_{r}^{l_r}
 \cdot w_1^{m_1}\cdots w_s^{m_{s}}.
 \end{equation*}
 in $\mathcal{C}_{n,r,s}$ by
 \begin{equation}\label{eq-intro-evaluate-chernpolynomial}
 \int_{X}c_1(T_X)^{k_1}\cdots c_{n}(T_X)^{k_n}\cdot c_1(E)^{l_{1}}\cdots 
 c_{r}(E)^{l_{r}}\cdot c_1(F)^{m_{1}}\cdots c_{s}(F)^{m_{s}}.
 \end{equation}
This evaluation extends to $\mathcal{C}_{n,r,s}$ linearly. For $f\in \mathcal{C}_{n,r,s}$, we denote the corresponding integrand by $\Phi_f(E,F)$. 
\begin{theorem}\label{thm-intro-universalSeries}(=Corollary \ref{cor-universalSeries})
Let $\Bbbk$ be a field of characteristic 0.
Suppose given natural numbers $d$, $r_1$ and $r_2$. Then 
\begin{enumerate}
	\item[(i)] there exists a series of polynomials $f_{i,j,k}\in \mathcal{C}_{d,r,s}$, for $i\geq 1$ and $j,k\geq 0$, such that for any smooth proper algebraic space $X$ of pure dimension $d$ and vectors bundles $E$ and $F$, with $\mathrm{rank}(E)=r$ and $\mathrm{rank}(F)=s$, we have
	\begin{equation}\label{eq-intro-universalSeries-1}
		1+\sum_{n=1}^{\infty}\chi(\Lambda_{-u}E^{[n]},\Lambda_{-v}F^{[n]})Q^n
		=\exp\Big(\sum_{i=1}^{\infty}\sum_{j=0}^{\infty}\sum_{k=0}^{\infty} Q^i u^j v^k 
		\int_X \Phi_{f_{i,j,k}}(E,F)
		\Big);
	\end{equation}
	\item[(ii)] given a series of polynomials $f_{i,j,k}\in \mathcal{C}_{d,r,s}$, for $i\geq 1$ and $j,k\geq 0$, to verify that (\ref{eq-intro-universalSeries-1}) holds for all smooth proper algebraic space $X$ of pure dimension $d$ over $\Bbbk$ and vector bundles $E$ and $F$, with $\mathrm{rank}(E)=r$ and $\mathrm{rank}(F)=s$, it suffices to verify it for all triples $(X,E,F)$ of the form
	\begin{equation}\label{eq-test-triples}
	\begin{cases}
		X=\mathbb{P}^{d_1}\times\cdots\times \mathbb{P}^{d_l},\ d_1+\dots+d_l=d,\\
		E=\bigoplus_{i=1}^l \pi_i^* L_i,\ 
		F=\bigoplus_{i=1}^l \pi_i^* M_i,\ \mbox{with}\ L_i\ \mbox{and}\ M_i \in 
		\{0, \mathcal{O}_{\mathbb{P}^d_i}, \mathcal{O}_{\mathbb{P}^d_i}(1)\}\\
		\mbox{such that}\ \mathrm{rank}(E)=r\ \mbox{and}\ \mathrm{rank}(F)=s.
	\end{cases}	
	\end{equation}
\end{enumerate}
\end{theorem}

In fact in (ii) it suffices to verify for a smaller set of triples $\mathcal{P}_{d,r,s}$ defined in Section \ref{sec:alg-cobordism}. 

For Hilbert schemes of points on surfaces, one can deduce Theorem \ref{thm-intro-universalSeries} from the existence of  \emph{universal polynomials for tautological integrals}  \cite[Theorem 4.1]{EGL01}. In higher dimensions there is a similar theorem of Rennemo  (\cite[Theorem 1.1]{Ren17}) on the integrals of the Chern classes of tautological sheaves and the Chern–Schwartz–MacPherson classes or Chern-Mather classes of the Hilbert scheme of points. But Rennemo's theorem  does NOT imply Theorem \ref{thm-intro-universalSeries}, because the Todd classes of singular schemes, which appear in the Riemann-Roch theorem for singular schemes, cannot be expressed in a universal way as a combination of Chern–Schwartz–MacPherson classes or Chern-Mather classes. The reason is well-known: the algebraic Euler characteristics are invariant in a flat family, while the topological ones are not. Another reason is, that the various types of singular Chern classes depend only on the reduced structure of the underlying schemes, which is not true for the Euler characteristics of coherent sheaves. Consequently, Rennemo's proof, which is based on  locally closed stratifications of Hilbert schemes, does not work for Theorem \ref{thm-intro-universalSeries}.
\\

Now we introduce an application of Theorem \ref{thm-intro-universalSeries}. The following is a conjecture proposed by Jian Zhou (see \cite{WZ14}). For surfaces it is shown by Wang-Zhou \cite{WZ14}; the case $u=0$ is a direct consequence of \cite[Theorem 2.4.5]{Sca09} (see also \cite{Kru18}). The conjecture fails for curves, as observed by Krug \cite{Kru18}). See also \cite{Wan16} for related problems for curves.
\begin{conjecture}\label{conj-1}
Let $X$ be a smooth projective scheme over $\Bbbk$ of dimension $d\geq 2$. Let $K,L$ be line bundles on $X$. Then
\begin{equation}\label{eq-conj-1}
1+\sum_{n=1}^{\infty}\chi(\Lambda_{-u}K^{[n]},\Lambda_{-v}L^{[n]})Q^n=
\exp \Big(\sum_{r=1}^{\infty}\chi(\Lambda_{-u^r}K, \Lambda_{-v^r}L)\frac{Q^r}{r}\Big).
\end{equation}
\end{conjecture}
We observe that the terms
\[
\chi(\Lambda_{-u^r}K, \Lambda_{-v^r}L)
\]
in the parentheses on the right side of (\ref{eq-conj-1}) is a cobordism invariant for the triples $(X,K,L)$. So one expects naturally the logarithmic of the left handside of (\ref{eq-conj-1}) is also a cobordism invariant. This motivates the statement of Theorem \ref{thm-intro-inc-exc-principle}. As an application of Theorem \ref{thm-intro-universalSeries}, we have:
\begin{corollary}\label{cor-intro-reduction-to-toric}
Let $\Bbbk$ be a field of characteristic zero.
Assume that Conjecture \ref{conj-1} holds for $d$-dimensional products of projective spaces  and line bundles $K$ and $L$ of the form (\ref{eq-test-triples}), then it holds for all equidimensional smooth proper algebraic spaces $X$ of dimension $d$ over $\Bbbk$ and line bundles $K,L$ on $X$.
\end{corollary}
As we observed after Theorem \ref{thm-intro-universalSeries}, it suffices to assume that Conjecture \ref{conj-1} holds for a slightly smaller class $\mathcal{P}_{d,1,1}$. In \cite{Hu21}, we further reduced Conjecture \ref{conj-1} to a conjectural identity on the equivariant Hilbert functions of the torus fixed points of $\mathrm{Hilb}^n(\mathbb{A}^d)$ (see \cite[Conjecture 5.5 and Proposition 5.7]{Hu21}), and then in \cite[Proposition 5.8]{Hu21} we  verified Conjecture \ref{conj-1} for equivariant line bundles on smooth proper toric 3-folds and  $n\leq 6$. So we obtain:
\begin{corollary}\label{cor-intro-3fold-7points}
For 3-dimensional smooth proper algebraic spaces $X$ over $\Bbbk$ of characteristic zero and line bundles $K,L$ on $X$, Conjecture \ref{conj-1} holds modulo $Q^7$.
\end{corollary}
For vector bundles of ranks greater than one, not much is known for the precise universal polynomials (but  \cite{Kru18}). See also Remark \ref{rmk-higher-ranks}.\\

Theorem \ref{thm-intro-universalSeries} is essentially equivalent to  Theorem \ref{thm-intro-inc-exc-principle} if in the statement of Theorem \ref{thm-intro-inc-exc-principle} only open subset $C$ of $\mathbb{P}^1$ are allowed to be the base curve. But up to the present the algebraic cobordism theory need the representatives to be projective schemes; in fact, although the set-up of \cite{LevP09} can be directly generalized to all smooth proper algebraic spaces, some key results of \cite{LevM07} are still not bypassed, for which the projectivity assumption seems not easy to drop. Therefore in the proof of Theorem \ref{thm-intro-universalSeries}, we  use Moishezon's theorem for algebraic spaces to reduce the problem to projective schemes. For the lack of a reference, we provide a proof in Appendix \ref{sec:Moishezon-thm}, following basically the argument of \cite[Theorem 2.2.16]{MM07}. The result herein is well-known to experts. Nevertheless the notion of C-projective morphisms (Definition \ref{def-CProj}) might be useful in other circumstances, even for schemes. \\

This paper is organized as follows. In Section \ref{sec:goodDegeneration} we recall the theory of good degenerations $\mathcal{I}_{\mathfrak{X}/C}^n$  of Hilbert schemes of points, and generalize the definition and the properness to algebraic spaces. A reader only interested in the projective case can skip Section \ref{sec:properness-algeSpace}, while Proposition \ref{prop-semisstablereduction-zero-dim-subschemes} may have some independent interests even in the projective case. In Section \ref{sec:flatness} we show the flatness of  $\mathcal{I}_{\mathfrak{X}/C}^n$. 
In Section \ref{sec:decompositionCentralFiber} we study the decomposition of the central fiber of  good degenerations. In Section \ref{sec:baseChange} we generalize the local constant property of sheaves flat over a base, to tame Deligne-Mumford stacks.
In Section \ref{sec:alg-cobordism} we recall the theory of algebraic cobordism of bundles,  state its generalization to a list of bundles, and give sketch  proof. In Section \ref{sec:tautological-sheaves-degeneration-formula}, from the results of the previous sections we prove the theorems stated in this introduction. In Appendix \ref{sec:Moishezon-thm}, as mentioned above, we show Moishezon's theorem for algebraic spaces. In Appendix \ref{sec:identity-chern-doublepoint} we show an identity involving Chern classes of the fibers of a double point degeneration; this identity can be regarded as a cycle version of \cite[Proposition 5]{LeeP12}, and is used in Section \ref{sec:alg-cobordism} to show that integrations of Chern polynomials factor through the double point relations.\\

To conclude the introduction, we make some comments on the characteristic zero assumption of the base field $\Bbbk$. It appears in two places. One appearance is in the use of algebraic cobordism. The main theorems (\cite[chapter 1]{LevM07}) of algebraic cobordism need $\mathrm{char}(\Bbbk)=0$, due to certain Bertini arguments, and the use of  resolution of singularities and  weak factorizations. To drop the assumption  in this appearance is of course difficult but I expect that finally one can make it. Another appearance is our use of the stack $\mathcal{I}_{\mathfrak{X}/C}^n$. For any given positive characteristic $p$, $\mathcal{I}_{\mathfrak{X}/C}^n$ is not tame  when $n$ is sufficiently large, 
thus the Euler characteristic of a tautological sheaf on $\mathcal{I}_{\mathfrak{X}/C}^n$  might not be well defined.
This issue seems more serious, and I do not know how to solve it. \\

\emph{Acknowledgement}:
I am grateful to Jian Zhou for sharing his insights on Conjecture \ref{conj-1}.
I thank  Huazhong Ke, Y.-P. Lee, Rahul Pandharipande, Feng Qu, Lei Song, Zhilan Wang, and especially Yu-Jong Tzeng for very helpful discussions in various aspects. I am also grateful to an anonymous referee for pointing out  inaccuracies and typos in a previous version of this paper.
This work is supported by Science and Technology Projects 202201010793 in Guangzhou, NSFC 11701579, and NSFC 12371063.

\section{Good degeneration of Hilbert schemes of points revisited}\label{sec:goodDegeneration}

The good degeneration of Hilbert scheme of points, and more generally of the Quot schemes, are constructed in \cite{LW15}. 
In this section we give a brief account of the GIT construction of \cite{GHH19} of the good degenerations of Hilbert schemes of points. In this section, our main concern is the properness of the good degeneration. We interpret the properness  as a kind of semistable reduction, which does not involve expanded degenerations. This enables us,  for later use, to generalize the theory to algebraic spaces. 

We fix a base field $\Bbbk$. 

\subsection{Good degeneration of Hilbert schemes of points on smooth projective schemes}
\begin{definition}\label{def-simpleDeg-chain}
Let $C$ be a  smooth curve over $\Bbbk$  with a distinguished $\Bbbk$-point $0\in C$. 
A \emph{simple  degeneration of chain type of length $l$} is a  proper and surjective morphism $\pi:\mathfrak{X}\rightarrow C$ such that 
\begin{enumerate}
	\item[(i)] $\mathfrak{X}$ is a smooth algebraic space over $\Bbbk$;
	\item[(ii)] the restriction of $\pi$ to $\pi^{-1}(C\setminus\{0\})$  is smooth;
	\item[(iii)] $\pi^{-1}(0)=Y_1\cup\dots\cup Y_{l}$ is a  divisor of $\mathfrak{X}$ with $l$ smooth irreducible components $Y_i$ for $1\leq i\leq l$, such that $Y_j$ intersects transversally with $Y_{j+1}$ at a smooth divisor $D_j$ of $Y_j$ and $Y_{j+1}$, and $Y_j\cap Y_k=\emptyset$ if $|j-k|\geq 2$, for $1\leq j,k\leq l$.
\end{enumerate}  
\end{definition}
The following definition stems from \cite[Def. 1.16]{GHH19}.
\begin{definition}\label{def-simpleDeg-bipartite}
Let $C$ be a  smooth curve over $\Bbbk$  with a distinguished $\Bbbk$-point $0\in C$. 
A \emph{bipartite simple  degeneration} is a  proper and surjective morphism $\pi:\mathfrak{X}\rightarrow C$ such that 
\begin{enumerate}
	\item[(i)] $\mathfrak{X}$ is a smooth algebraic space over $\Bbbk$;
	\item[(ii)] the restriction of $\pi$ to $\pi^{-1}(C\setminus\{0\})$  is smooth;
	\item[(iii)] $\pi^{-1}(0)$ is a divisor of $\mathfrak{X}$, and there are two sub-divisors labelled as $Y_1$ and $Y_2$, such that $\pi^{-1}(0)=Y_1\cup Y_2$, and $Y_1$ and $Y_2$ are smooth divisors of $\mathfrak{X}$, and $Y_1$ intersects transversally with $Y_{2}$ at a smooth divisor $D$ of $Y_1$ and $Y_{2}$. It is allowed that $D=\emptyset$. 
\end{enumerate}  
\end{definition}

If we ignore the choice of the labelling $Y_1$ and $Y_2$, a bipartite simple  degeneration is no other than  a \emph{double point degeneration} in the sense of \cite{LevP09}. Note that we do not require $Y_1$, $Y_2$, or $D$ to be irreducible, so a simple degeneration of chain type can always be made into a bipartite simple degeneration, by grouping the irreducible components $Y_i$'s into two smooth divisors renamed by $Y'_1$ and $Y'_2$. In particular, if $l=1$,  we can take $Y_2=D=\emptyset$.


As usual we have required that a degeneration $\pi$ is proper. To facilitate some intermediate constructions we introduce the following definition.
\begin{definition}
We call a surjective morphism $\pi:\mathfrak{X}\rightarrow C$  satisfying (i)-(iii) of Definition \ref{def-simpleDeg-chain} a \emph{local} simple   degeneration of chain type (of length $l$). Similarly, dropping the properness requirement of Definition \ref{def-simpleDeg-bipartite}, we have the notion of \emph{local bipartite simple degeneration}.
\end{definition}

Let  $\pi:\mathfrak{X}\rightarrow C$ be a local simple degeneration of chain type. 
Since we concern only the fibers over a neighborhood of $0\in C$, we can assume that there is given  an étale morphism $\tau:C\rightarrow \mathbb{A}^1$ such that $\tau^{-1}(0)=0$. 
Since the closed points of $D$ whose residue fields are separable over $\Bbbk$ are dense in $D$, from the Jacobian criterion we know that étale locally near a point $w$ of $D$, we have a local model of $\mathfrak{X}$ of the form $\Spec\ \Bbbk[x,y,z,\dots,t]/(xy-t)$. More precisely, there is a commutative diagram
\begin{equation}\label{graph-local-model-1}
\xymatrix{
\Spec\ \Bbbk[x,y,z,\dots,t]/(xy-t) \ar[d]_{\pi'}  & \mathfrak{W} \ar[r] \ar[l]_>>>>>>{f} & \mathfrak{X} \ar[d]^{\pi} \\
	\Spec\ \Bbbk[t]  && C  \ar[ll]_>>>>>>>>>>>>>>>>>>>{\tau}
}	
\end{equation}
where $\mathfrak{W}$ is a scheme smooth over $\Bbbk$ and $\mathfrak{W}\rightarrow \mathfrak{X}$ is étale, and $\pi'$ is the projection to the last coordinate, and $f$ is étale. Then it will turn out that in most of the following constructions on $\mathfrak{X}$ one can essentially deal with only the case that $C$ is an open subset of $\mathbb{A}^1$ containing $0$.

Let $G[n]\subset \mathrm{SL}_{n+1}$ be the diagonal maximal torus, and let $G[n]$ act on $\mathbb{A}^{n+1}$ as 
\begin{equation}\label{eq-action-1}
	(\sigma_1,\dots,\sigma_{n+1}).(t_1,\dots,t_{n+1})=(\sigma_1 t_1,\dots,\sigma_{n+1}t_{n+1}),
\end{equation}
and act on $\mathbb{A}^1$ trivially, so that the multiplication morphism $\mu:\mathbb{A}^{n+1}\rightarrow \mathbb{A}^1$ is equivariant. Define $C[n]=C\times_{\mathbb{A}^1}A^{n+1}$, which is induced by the morphisms $t$ and $\mu$.  We have a canonical projection $C[n]\rightarrow C$, and induced actions of $G[n]$ on $C[n]$.

\begin{definition}[Expanded degeneration]
Let $\mathfrak{X}\rightarrow \mathbb{A}^1$ be a local bipartite simple degeneration. The \emph{expanded degeneration} $\mathfrak{X}[n]$ is a small resolution of $\mathfrak{X}\times_{\mathbb{A}^1}\mathbb{A}^{n+1}$ defined  recursively as follows.
\begin{enumerate}
	\item $\mathfrak{X}[0]=\mathfrak{X}$;
	\item Suppose that the resolution $\rho_{n-1}:\mathfrak{X}[n-1]\rightarrow \mathfrak{X}\times_{\mathbb{A}^1}\mathbb{A}^{n}$ is constructed. Composing $\rho_{n-1}$ with the projection to the last component of $\mathbb{A}^n$, we regard $\mathfrak{X}[n-1]$ as a scheme over $\mathbb{A}^1$ and form the fiber product $\mathfrak{X}[n-1]\times_{\mathbb{A}^1}\mathbb{A}^2$, where the morphism
	 $m:\mathbb{A}^2\rightarrow \mathbb{A}^1$ is the multiplication morphism.  Then	$\mathfrak{X}[n]$ is defined as the blow-up of the fiber product
	\[
	\xymatrix@C=4pc{
		\mathfrak{X}[n-1]\times_{\mathbb{A}^1}\mathbb{A}^2 \ar[r]^{(\rho_{n-1},m)} \ar[d]_{\mathrm{pr}_1} & \mathfrak{X}\times_{\mathbb{A}^1}\mathbb{A}^{n+1} \ar[d] & (x,t_1,\dots,t_{n+1}) \ar@{|->}[d] \\
		\mathfrak{X}[n-1] \ar[r]^{\rho_{n-1}} & \mathfrak{X}\times_{\mathbb{A}^1}\mathbb{A}^n & (x,t_1,\dots,t_{n-1},t_n t_{n+1})
	}
	\]
	along the strict transform of $Y_1\times V(t_{n+1})\subset \mathfrak{X}\times_{\mathbb{A}^1}\mathbb{A}^{n+1}$ under the partial resolution $(\rho_{n-1},m)$.
\end{enumerate}
\end{definition}

Let $\pi:\mathfrak{X}\rightarrow C$ be a local bipartite simple degeneration. The associated expanded degeneration $\mathfrak{X}[n]$ is a small $G[n]$-equivariant degeneration of $\mathfrak{X}\times_C C[n]=\mathfrak{X}\times_{\mathbb{A}^1}\mathbb{A}^{n+1}$. We denote the induced projection $\mathfrak{X}[n]\rightarrow C[n]$  by $\pi[n]$.

\begin{example}
Consider the local model $\mathfrak{X}=\Spec\ \Bbbk[x,y,z,\dots,t]/(xy-t)$. The fiber product $\mathfrak{X}\times_{\mathbb{A}^1}\mathbb{A}^2$ is $\Spec\ \Bbbk[x,y,z,\dots,t_1,t_2]/(xy-t_1 t_2)$. It is a cone over a quadric in $\mathbb{P}^2$. Blowing up the origin of the cone we  obtain a resolution, say $\mathfrak{X}'$. But this resolution is not the one we want; in the projection $\mathfrak{X}'\rightarrow \mathbb{A}^2\xrightarrow{m}\mathbb{A}^1$, the fibers $\mathfrak{X}'_t$ are not isomorphic to $\mathfrak{X}_t$  for $0\neq t\in \mathbb{A}^1$. So we take a small resolution, the blow-up along $\{y=t_2=0\}$. Now the local model of $\mathfrak{X}[1]$ is the closed subscheme of 
\begin{equation*}
	\mathfrak{X}\times \mathbb{A}^{2}\times \mathbb{P}^1=\Spec\ \Bbbk[x,y,z,\dots,t_1,t_2]\times \mathrm{Proj}\ \Bbbk[u_1,v_1]
\end{equation*}
defined by
\[
\begin{cases}
u_1 x=v_1 t_1,\\
v_1 y=u_1 t_2,\\
xy=t_1 t_2.
\end{cases}
\]
In fact the equation $xy=t_1 t_2$ is redundant. Replace $t_2$ by $t_2 t_3$, and blow-up the ideal $(y,t_3)$, we obtain the local model of $\mathfrak{X}[2]$ as the closed subscheme of
\begin{equation*}
	\mathfrak{X}\times \mathbb{A}^{3}\times (\mathbb{P}^1)^2=\Spec\ \Bbbk[x,y,z,\dots,t_1,t_2,t_3]\times \prod_{i=1}^{2} \mathrm{Proj}\ \Bbbk[u_i,v_i]
\end{equation*} 
defined by
\[
\begin{cases}
u_1 x=v_1 t_1,\\
u_2 v_1=v_2 u_1 t_2,\\
v_2 y=u_2 t_3.
\end{cases}
\]
\end{example}
Repeating this process one obtains the following local description of $\mathfrak{X}[n]$ (\cite[Proposition 1.7]{GHH19}, see also \cite[Lemma 1.2]{LiJ01}).

\begin{proposition}\label{prop-doublepoint-degeneration-local-model}
Let $\mathfrak{X}=\Spec\ \Bbbk[x,y,z,\dots]$, and $\pi$ be the local bipartite simple degeneration $\mathfrak{X}\rightarrow C=\mathbb{A}^1=\Spec\ \Bbbk[t]$ induced by $t\mapsto xy$. Consider the product 
\begin{equation}\label{eq-local-chart}
	\mathfrak{X}\times \mathbb{A}^{n+1}\times (\mathbb{P}^1)^n=\Spec\ \Bbbk[x,y,z,\dots,t_1,\dots,t_{n+1}]\times \prod_{i=1}^n\mathrm{Proj}\ \Bbbk[u_i,v_i].
\end{equation}
\begin{enumerate}
	\item[(i)] $\mathfrak{X}[n]$ is the closed subscheme of $\mathfrak{X}\times\mathbb{A}^{n+1}\times (\mathbb{P}^1)^n$ defined by 
	\begin{eqnarray}\label{eq-local-equations}
	\begin{cases}
	u_1 x=v_1 t_1,\\
	u_i v_{i-1}=v_i u_{i-1} t_i,\ 1<i\leq n,\\
	v_n y=u_n t_{n+1}.
	\end{cases}
	\end{eqnarray}
	\item[(ii)] The $G[n]$-action on $\mathfrak{X}[n]$ is the restriction of the action which is trivial on $\mathfrak{X}$, given by (\ref{eq-action-1}) on $\mathbb{A}^{n+1}$, and given by
	\[
	(\sigma_1,\dots,\sigma_{n+1}).[u_i,v_i]=[\sigma_1 \sigma_2\cdots \sigma_i u_i:v_i]
	\]
	on the $i$-th copy of $\mathbb{P}^1$.
\end{enumerate}
\end{proposition}\pqed

\begin{definition}
For 
$\mathbf{a}=(a_0,a_1,\dots,a_{r},a_{r+1})\in \mathbb{Z}^{r+2}$
satisfying
\[
1=a_0\leq a_1<\dots<a_i<\dots<a_r\leq a_{r+1}=n+1,
\]
we define
\[
I_{\mathbf{a}}=\{a_1,\dots,a_r\}\subset [n+1]=\{1,\dots,n+1\},
\]
and
\[
\mathbf{v}_{\mathbf{a}}=(a_1-a_0,\dots,a_{r+1}-a_r)\in \mathbb{Z}^{r+1}.
\]
\end{definition}\label{prop-fiber-over-stratum}
The following is a description of a general fiber over a stratum of $C[n]$, see \cite[Lemma 2.2]{LW15} or \cite[Proposition 1.12]{GHH19}.
\begin{proposition}
A general fiber over the locus $\{t_{a_1}=\dots=t_{a_r}=0\}\subset C[n]$ has a dual graph as
\[
\bullet\xrightarrow{a_1}\circ \xrightarrow{a_2}\circ\rightarrow\cdots\xrightarrow{a_r}\bullet,
\]
with components $\Delta_I^{a_0},\dots,\Delta_I^{a_r}$, such that
\begin{enumerate}
	\item[(i)] $D\cong \Delta_I^{a_i}\cap \Delta_{I}^{a_{i+1}}$ for $0\leq i\leq r-1$, and  any triple intersection of $\Delta_I^{a_0},\dots,\Delta_I^{a_r}$ is empty;
 	\item[(ii)] $\Delta_I^{a_0}\cong Y_1$, $\Delta_{I}^{a_r}\cong Y_2$;
 	\item[(iii)] For $1\leq i\leq r-1$, $\Delta_I^{a_i}$ is isomorphic to the projective bundle $\mathbb{P}(N_{D/Y_1}\oplus \mathcal{O}_D)$, where the base $D$ is located in the bundle either as $D\cong \Delta_I^{a_{i-1}}\cap \Delta_I^{a_i}$ or $D\cong \Delta_I^{a_{i}}\cap \Delta_I^{a_{i+1}}$.
 \end{enumerate}  
\end{proposition}

Now let $\pi:\mathfrak{X}\rightarrow C$ be a bipartite simple degeneration.
Let $\mathbf{H}^n(\mathfrak{X}/C)=\mathrm{Hilb}^n(\mathfrak{X}[n]/C[n])$, the Hilbert functor of $n$ points  associated with $\pi[n]$. By \cite[Proposition 1.9, Theorem 4.4]{Ryd11}, $\mathbf{H}^n(\mathfrak{X}/C)$ is represented by an algebraic space separated of finite type over $C[n]$. Then by \cite[Theorem 1.1]{OS03}, $\mathbf{H}^n(\mathfrak{X}/C)$ is an algebraic space proper over $C[n]$.

 We denote by $\mathbf{H}^n_{sm}(\mathfrak{X}/C)$ the open subspace of $\mathbf{H}^n(\mathfrak{X}/C)$, which parametrizes the length $n$ closed subspaces of the fibers of $\pi[n]$ that are supported on the smooth locus of the fibers.

 For any point $[Z]\in \mathbf{H}^n_{sm}(\mathfrak{X}/C)$, let 
 \[
 I_{[Z]}=\big\{i\in [n+1]:t_i|_{\mathrm{Supp}(Z)}=0\big\}.
  \]
  Suppose $\mathbf{a}\in \mathbb{Z}^{r+2}$ such that $I_{\mathbf{a}}=I_{[Z]}$. Then $[Z]$ can be written as 
 \[
 [Z]=\bigsqcup_{i=0}^{r}[Z_i]
 \]
 such that $Z_i$ lies in $\Delta_I^{a_i}$ (and not in  other $\Delta_I^{a_j}$ for $j\neq i$), for $0\leq i\leq r$. Define
 \[
 \mathbf{v}(Z)=\big(\mathrm{length}(Z_0),\dots,\mathrm{length}(Z_r)\big)\in \mathbb{Z}^{r+1}.
 \]
\begin{proposition}\label{thm-stable-locus}\emph{(\cite[Theorem 2.10]{GHH19})}
Let $\mathfrak{X}$ be a smooth quasi-projective scheme and $\pi:\mathfrak{X}\rightarrow C$ a bipartite simple  degeneration. Then there exists a $G[n]$-linearized relative ample line bundle $\mathcal{M}$ on $\mathbf{H}^n(\mathfrak{X}/C)$, such that a point $[Z]\in \mathbf{H}^n(\mathfrak{X}/C)$ is semistable if and only if the following two conditions are satisfied:
\begin{enumerate}
 	\item[(i)] $[Z]\in \mathbf{H}^n_{sm}(\mathfrak{X}/C)$;
 	\item[(ii)] $\mathbf{v}(Z)=\mathbf{v}_{\mathbf{a}}$ where $\mathbf{a}\in \mathbb{Z}^{r+2}$ is defined such that $I_{\mathbf{a}}=I_{[Z]}$.
 \end{enumerate}
Moreover, when this happens, $[Z]$ is a stable point.
\end{proposition}
Thus in case that $\mathfrak{X}$ is projective,  for appropriate $G[n]$-linearized relative ample line bundles $\mathcal{M}$ as in Theorem \ref{thm-stable-locus}, the (semi)stable locus of $\mathbf{H}^n(\mathfrak{X}/C)$ is independent of the choice of such $\mathcal{M}$. This motivates the following definition for algebraic spaces.

\begin{definition}\label{def-good-degeneration-Hilbert-scheme-points}
Let $\mathfrak{X}$ be a smooth separated algebraic space and $\pi:\mathfrak{X}\rightarrow C$ a bipartite simple  degeneration. We say a point $[Z]\in \mathbf{H}^n(\mathfrak{X}/C)$ is \emph{stable} if and only if the  conditions (i) and (ii) in Proposition \ref{thm-stable-locus} are satisfied. We denote the open subspace of stable points by $\mathbf{H}^n_{\st}(\mathfrak{X}/C)$.

\end{definition}

\begin{definition}
The stacky quotient 
\[
\mathcal{I}_{\mathfrak{X}/C}^n=[\mathbf{H}^n_{\st}(\mathfrak{X}/C)/G[n]]
\]
is called the \emph{good degeneration} of (the general fibers of) $\mathrm{Hilb}^n(\mathfrak{X}/C)$.
\end{definition}

The following theorem is \cite[Theorem 4.14]{LW15} and \cite[Theorem 3.2]{GHH19}.
\begin{theorem}\label{thm-properness-X-projective}
Let $\mathfrak{X}$ be a smooth quasi-projective scheme and $\pi:\mathfrak{X}\rightarrow C$ a bipartite simple degeneration. Then
$\mathcal{I}_{\mathfrak{X}/C}^n$ is a Deligne-Mumford stack proper over $C$, and has a projective coarse moduli scheme.
\end{theorem}

\begin{remark}\label{rmk-baseField}
Both \cite{LW15} and \cite{GHH19} assume that the base field $\Bbbk$ is algebraically closed and $\mathrm{char}(\Bbbk)=0$. Both assumptions are redundant.

In fact, by Definition \ref{def-good-degeneration-Hilbert-scheme-points}, for any field extension $\Bbbk\subset K$ it holds that $\mathcal{I}_{\mathfrak{X}_K/C_K}^n\cong \mathcal{I}_{\mathfrak{X}/C}^n\times_k K$, as the usual base change of Hilbert schemes. So the statement over arbitrary base field $\Bbbk$ follows from that over an algebraically closed one. Since the group $G[n]\cong \mathbb{G}_m^n$ is linearly reductive in all characteristics, the GIT approach in \cite{GHH19} holds in positive characteristics (probably also the approach of \cite{LW15}). 
\end{remark}

\subsection{Good degeneration of Hilbert schemes of points on smooth proper algebraic spaces}\label{sec:properness-algeSpace}

In the following of this section we generalize the properness statement in  Theorem \ref{thm-properness-X-projective} to algebraic spaces. 
Our strategy is to reduce the problem to the projective case. For this we need  an equivalent statement (Proposition \ref{prop-semisstablereduction-zero-dim-subschemes}) of the properness. 
We begin with several constructions. For ease of notations, for a scheme $X$ and a quasi-coherent sheaf of ideals $\mathcal{I}$ generated by regular functions $h_1,\dots,h_k$ on $X$, we let $V(h_1,\dots,h_k)=V(\mathcal{I})$, the closed subscheme defined by $\mathcal{I}$.

\begin{construction}\label{cons-specialization-0dim-subscheme-1}
Let $\pi:\mathfrak{Y}\rightarrow C$ be a local simple degeneration of chain type of length $l$, and $\tau:C\rightarrow \mathbb{A}^1=\Spec\ \Bbbk[t]$ an étale morphism such that $\tau^{-1}(0)=0$ as the beginning of this section. The $l$ components of the fiber $\pi^{-1}(0)$ has been labelled $Y_1,\dots,Y_l$.  Let $m$ be a natural number. Let $h_m:\mathbb{A}^1\rightarrow \mathbb{A}^1$ be the morphism induced by $t\mapsto t^m$. Consider the morphism $\mathfrak{Y}_m\rightarrow C_m$ defined by the following cartesian diagram:
\[
\xymatrix{
	\ar @{} [dr] |{\square} \mathfrak{Y}_m \ar[r]^{\pi_m} \ar[d]_{f_m} & \ar @{} [dr] |{\square} C_m \ar[r]^{\tau_m} \ar[d]_{g_m} & \mathbb{A}^1 \ar[d]^{h_m} &  \\
	\mathfrak{Y} \ar[r]^{\pi} & C \ar[r]^{\tau} & \mathbb{A}^1 & .
}
\]
Then $C_m$ is a regular curve, and away from $0=\tau_m^{-1}(0)$, $\pi_m$ is smooth.
The output of this construction is $\tilde{\pi}_m:\widetilde{\mathfrak{Y}_m}\rightarrow C_m$ defined as follows.

 If $m=1$ we set $\widetilde{\mathfrak{Y}_1}=\mathfrak{Y}_1=\mathfrak{Y}$. Suppose $m>1$. Let $t'=(\tau_m\circ \pi_m)^*(t)$. We successively blow up $\mathfrak{Y}_m$ along $f^{-1}(Y_2)\cap V(t'^{m-1})$, the strict transform of $f^{-1}(Y_2)\cap V(t'^{m-2})$, \dots, and  the strict transform of $f^{-1}(Y_2)\cap V(t')$; then we successively blow up  along the strict transform of  $f^{-1}(Y_3)\cap V(t'^{m-1})$, the strict transform of $f^{-1}(Y_3)\cap V(t'^{m-2})$, \dots, and  the strict transform of $f^{-1}(Y_3)\cap V(t')$; \dots; we successively blow up  along the strict transform of  $f^{-1}(Y_l)\cap V(t'^{m-1})$, the strict transform of $f^{-1}(Y_l)\cap V(t'^{m-2})$, \dots, and finally the strict transform of $f^{-1}(Y_l)\cap V(t')$.  Denote the resulted space by $\widetilde{\mathfrak{Y}_m}$.  When $\pi$ is the local model $\Spec\ \Bbbk[x,y,\dots,t]/(xy-t)\rightarrow \Spec\ \Bbbk[t]$,  $\widetilde{\mathfrak{Y}_m}$  is the subscheme of 
\[
\Spec\ \Bbbk[t',x,y,\dots]\times \prod_{i=1}^{m-1}\mathrm{Proj}\ \Bbbk[u_i,v_i]
\]
defined by
\begin{eqnarray}\label{eq-local-equations-3}
	\begin{cases}
	u_1 x=v_1 t',\\
	u_i v_{i-1}=v_i u_{i-1} t',\ 1<i\leq m-1,\\
	v_{m-1} y=u_{m-1} t'.
	\end{cases}
\end{eqnarray}
It follows that $\widetilde{\mathfrak{Y}_m}$ is regular, and $\tilde{\pi}_m:\widetilde{\mathfrak{Y}_m}\rightarrow C_m$ is a simple degeneration of chain type of length $(l-1)m+1$. we can perform this construction successively. The following lemma follows directly. 
\begin{lemma}\label{lem-construction-1-transitivity}
Let $\pi:\mathfrak{Y}\rightarrow C$ be a local simple degeneration of chain type. Let $m_1$ and $m_2$ be natural numbers. Let us apply Construction \ref{cons-specialization-0dim-subscheme-1} to $\pi$ with $m_1$, and then apply Construction \ref{cons-specialization-0dim-subscheme-1} to the resulted simple degeneration of chain type, with $m_2$. The resulted simple degeneration of chain type is naturally identified with the Construction \ref{cons-specialization-0dim-subscheme-1} applied to $\pi$ with $m_1 m_2$.
Namely,
\begin{equation}
	\widetilde{\big(\widetilde{\mathfrak{Y}_{m_1}}\big)_{m_2}}\cong \widetilde{\mathfrak{Y}_{m_1 m_2}}.
\end{equation}
\end{lemma}

\begin{remark}
Note that this construction depends only on  $\pi:\mathfrak{Y}\rightarrow C$ together with the distinguished point $0\in C$, and $m\in \mathbb{N}$, and is independent of the choice of the étale morphism $\tau:C\rightarrow \mathbb{A}^1$. 
\end{remark}



\end{construction}

\begin{construction}(\cite[2.8.1, 2.8.5]{EGA-IV})\label{cons-specialization-0dim-subscheme-2}
Let $R$ be a DVR,  and $K$ the fraction field of $R$. Let $\mathcal{Y}$ be an algebraic space  and  $\pi:\mathcal{Y}\rightarrow \Spec(R)$ a quasi-compact separated dominant morphism. Denote by $\iota:\mathcal{Y}_K\hookrightarrow \mathcal{Y}$ the inclusion of the generic fiber. Let $Z$ be a $0$-dimensional closed subspace of $\mathcal{Y}_K$. The image of $\mathcal{O}_{\mathcal{Y}}\rightarrow \iota_* \mathcal{O}_Z$ is the structure sheaf of a closed subspace of $\mathcal{Y}$, denoted by $\overline{Z}$. Then $\overline{Z}$ is flat over $R$ and $\overline{Z}_K\cong Z$; moreover $\overline{Z}$ is the unique closed subspace of $\mathcal{Y}$ that satisfies both properties. The underlying space of points $|\overline{Z}|$ is the closure of $|Z|$ in $\mathcal{Y}$ (\cite[9.5.4]{EGA-I}).
When $R$ and $\mathcal{Y}$ are essentially of finite type over a field $\Bbbk$ and $\pi$ is proper, $\overline{Z}$ is no other than the output of the valuative criterion for properness applied to $\mathrm{Hilb}^n\big(\mathcal{Y}/\Spec(R)\big)$ (e.g. \cite[the end of \S 2]{OS03}), where $n=\mathrm{length}(Z)$.
\end{construction}

A composition of Constructions \ref{cons-specialization-0dim-subscheme-2} and \ref{cons-specialization-0dim-subscheme-1} yields the following one.
\begin{construction}\label{cons-specialization-0dim-subscheme-3}
Suppose there is given  $\mathfrak{Y}\xrightarrow{\pi}C\xrightarrow{\tau}\mathbb{A}^1$ and $m\in \mathbb{N}$ as in Construction \ref{cons-specialization-0dim-subscheme-1}. Let $R$ a DVR, $\rho:\Spec(R)\rightarrow C$ a morphism such that $\rho$ sends the closed point of $\Spec(R)$ to $0\in C$, and the generic point to $C\setminus\{0\}$. Denote the fraction field of $R$ by $K$. Let $\mathcal{Y}=\mathfrak{Y}\times_{C}\Spec(R)$. Let $Z$ be a closed subscheme of $Y=\mathfrak{Y}\times_{C}\Spec(K)$ of length $n$. These form our input data.

Now applying Construction \ref{cons-specialization-0dim-subscheme-1} to $\mathfrak{Y}\xrightarrow{\pi}C\xrightarrow{\tau}\mathbb{A}^1$ and $m$, we obtain $\tilde{\pi}_m:\widetilde{\mathfrak{Y}_m}\rightarrow C_m$. Let $L$ (resp. $L_m$) be the field of rational functions on $C$ (resp. $C_m$). Let $K_m$ be a finite extension of $K$ such that  there is a commutative diagram
\begin{equation}\label{graph-cons-specialization-0dim-subscheme-1}
\xymatrix{
	K_m & L_m \ar[l] & \\
	K \ar[u] & L \ar[u]_{g_m^*} \ar[l]_>>>>>>>{\rho^*} & .
}
\end{equation}
Such $K_m$ exists; for example one can take $K_m$ as the residue field of one component of $K\otimes_L L_m$. Let $R_m$ be the localization at one closed point of the integral closure of $R$ in $K_m$. Then $R_m$ is also a DVR, and by the properness of $C_m\xrightarrow{g}C$, there is a commutative diagram completing (\ref{graph-cons-specialization-0dim-subscheme-1}):
\begin{equation}\label{graph-cons-specialization-0dim-subscheme-2}
 	\xymatrix{
 	\Spec(R_m) \ar[r]^>>>>>{\rho_m} \ar[d] & C_m \ar[d]^{g_m} & \\
 	\Spec(R) \ar[r]^>>>>>>{\rho} & C &,
 	}
 \end{equation} 
and $\rho_m$ sends the closed point of $\Spec(R_m)$ to $0=g_m^{-1}(0)\in C_m$. 
Let $\mathcal{Y}_m$ be the fiber product:
\begin{equation}\label{graph-cons-specialization-0dim-subscheme-3}
	\xymatrix{
	\ar @{} [dr] |{\square} 
	\mathcal{Y}_m \ar[r]^{\theta_m} \ar[d]_{p_m} & \widetilde{\mathfrak{Y}_m} \ar[d]^{\tilde{\pi}_m}& \\
	\Spec(R_m) \ar[r]^>>>>>{\rho_m} & C_m& .
	}
\end{equation}
Then $Z_m:=Z\times_K K_m$ is a closed subscheme of 
\[
Y_m:=\mathcal{Y}_m\times_{\Spec(R_m)}\Spec(K_m)
\] 
of length $n$. Applying the Construction \ref{cons-specialization-0dim-subscheme-2} to $\mathcal{Y}_m/\Spec(R_m)$ and $Z_m$ we obtain a closed subscheme $\mathcal{Z}_m\subset \mathcal{Y}_m$. We regard $p_m:\mathcal{Y}_m\rightarrow\Spec(R_m)$ and $\mathcal{Z}_m$ as the output of this construction. 

Note that there are choices for $K_m$ and $R_m$ in this construction, but we take an arbitrary one. For a fixed $m$, and a given choice of $K_m$ and $R_m$, this construction is functorial in $\mathfrak{Y}$. 
\end{construction}

\begin{proposition}\label{prop-semisstablereduction-zero-dim-subschemes}
Keep the notations of Construction \ref{cons-specialization-0dim-subscheme-3}.  Denote by $D_m$ the singular locus of the special fiber of $\tilde{\pi}_m:\widetilde{\mathfrak{Y}_m}\rightarrow C_m$. Then there exists $m$ such that $\mathcal{Z}_m$ does not meet $\theta_m^{-1}(D_m)$.
\end{proposition}

\begin{proof}
The statement is equivalent to that the closure of the underlying topological space of $Z_m$ does not meet $\theta_m^{-1}(D_m)$. Let $D$ be the singular locus of the special fiber of $\mathfrak{Y}\rightarrow C$.

\textsc{Step 0}: Since the construction on $\mathfrak{Y}$ depends only on the composition $\mathfrak{Y}\xrightarrow{\tau\circ \pi}\mathbb{A}^1$, we can assume $C$ to be an open subset of $\mathbb{A}^1$ containing $0$. 

\textsc{Step 1}: The statement is étale local on $\mathfrak{Y}$, in the following sense. Let $\{f_{\alpha}:\mathfrak{U}_\alpha\rightarrow \mathfrak{Y}\}_{\alpha\in A}$ be a finite family of étale morphisms of finite type, such that the union of their images cover $D$. Let $\pi_{\alpha}=\pi\circ f_{\alpha}$, and
\[
\mathcal{\phi}_{\alpha}:
\mathcal{U}_{\alpha}=\mathfrak{U}_{\alpha}\times_C \Spec(R)\rightarrow
\mathcal{Y}=\mathfrak{Y}\times_C \Spec(R).
\]
the morphism induced by $f_{\alpha}$. Let  $Z_{\alpha}=\phi_{\alpha}^{-1}(Z)$. For $k\geq 1$, denote by $\widetilde{f_{\alpha,k}}$ the morphism in the commutative diagram
\[
\xymatrix{
	\widetilde{(\mathfrak{U}_{\alpha})_k} \ar[dr] \ar[rr]^{\widetilde{f_{\alpha,k}}} &&
	\widetilde{\mathfrak{Y}_k} \ar[dl] \\
	& C_k
}
\]
 induced by $f_{\alpha}$. Let $E_{i,k}=\widetilde{f_{\alpha,k}}^{-1}(D_k)$. 
  We claim that, if the conclusion holds for each $\mathfrak{U}_{\alpha}\xrightarrow{\pi\circ f_{\alpha}}C$ and $Z_{\alpha}$, then it holds for $\mathfrak{Y}\rightarrow C$. 
  We show this claim in three sub-steps.
\begin{enumerate}
	\item[1)]
Suppose the conclusion holds for $\mathfrak{U}_{\alpha}\xrightarrow{\pi_\alpha}C$ with $m_{\alpha}$, and choices $K_{\alpha}$ of $K_{m_{\alpha}}$ (resp. $R_{\alpha}$ of $R_{m_{\alpha}}$),  for $\alpha\in A$.  Let 
\[
m=\prod_{\alpha\in A}m_{\alpha}.
\]
 Let $R_m$ be a DVR with a local homomorphism $R\rightarrow R_m$, such that it factors through $R\rightarrow R_{\alpha}$ for $\alpha\in 
 A$, and  the quotient field extension   $K\rightarrow K_m$ is finite. For example one can take $K_m$ to be the residue field of a component of
\[
\bigotimes_{\alpha\in A} K_{\alpha},\ \mbox{where the $\otimes$ is over } K,
\]
and take $R_m$ to be the localization of the integral closure of $R$ in $K_m$ at a closed point lying over  the closed point of $R_{\alpha}$ for $\alpha\in A$.
For each $\alpha\in A$, let $R'_{\alpha}$ be a ring such that $\Spec(R'_{\alpha})\cong\Spec(R_{\alpha})\times_{C_{m_{\alpha}}}C_m$, and let $\mathcal{U}'_{\alpha}=\Spec(R'_{\alpha})\times_{C_m}\widetilde{(\mathfrak{U}_{\alpha})_m}$. Consider the following commutative diagram
\begin{equation}
	\xymatrix{
	 \mathcal{U}'_{\alpha} \ar[dr]^{\gamma'_{\alpha}} \ar[dd] \ar[rr]^{\theta'_{\alpha}}  && \widetilde{(\mathfrak{U}_{\alpha})_m} \ar'[d][dd]^{\tilde{\pi}_{m}} \ar[dr]^{\gamma_{\alpha}} \\
	 & \mathcal{U}_{\alpha} \ar[dd] \ar[rr]^<<<<<<<<{\theta_{\alpha}} && \widetilde{(\mathfrak{U}_{\alpha})_{m_{\alpha}}} \ar[dd]^{\tilde{\pi}_{m_{\alpha}}}  \\
	 \Spec(R'_{\alpha}) \ar[dr] \ar'[r]^<<<<<<<{\rho'_{\alpha}}[rr] && C_m \ar[dr] \\
	 & \Spec(R_{\alpha}) \ar[rr]^{\rho_{\alpha}} && C_{m_{\alpha}}
	}
\end{equation}
where the front, the back, the top and the bottom squares are cartesian. By Construction \ref{cons-specialization-0dim-subscheme-1} and Lemma \ref{lem-construction-1-transitivity}, $\widetilde{\mathfrak{U}_m}$ is obtained by successive blowing-ups of  
$C_m\times_{C_{m_{\alpha}}}\widetilde{\mathfrak{U}_{m_{\alpha}}}$ along centers supported away from the nodes of the special fiber. Since $\rho_{\alpha}$, and thus $\theta_{\alpha}$, are flat, and the top square is cartesian, it follows that $\mathcal{U}'_{\alpha}$ is obtained by successive blowing-ups of  
$\Spec(R'_{\alpha})\times_{\Spec(R_\alpha)}\mathcal{U}_{\alpha}$ along centers supported over 
$\theta_{\alpha}^{-1}(E_{\alpha,m_{\alpha}})$ (\cite[Tag 080F]{StPr}). By the assumption, the closure $(\mathcal{Z}_{\alpha})_{m_{\alpha}}$ of $(Z_{\alpha})_{m_{\alpha}}$ in $\mathcal{U}_{\alpha}$ does not meet 
$\theta_{\alpha}^{-1}(E_{\alpha,m_{\alpha}})$. So the strict transform of $(\mathcal{Z}_{\alpha})_{m_{\alpha}}$ via $\gamma'_{\alpha}$ does not meet $\theta'^{-1}_{\alpha}(E_{\alpha,m})$. Let $\mathcal{Z}'_{\alpha}$ be the closure of 
$\gamma'^{-1}_{\alpha}\big((Z_{\alpha})_{m_{\alpha}}\big)$ in $\mathcal{U}'_{\alpha}$.  Since the strict transform $(\mathcal{Z}_{\alpha})_{m_{\alpha}}$ via $\gamma'_{\alpha}$ contains $\mathcal{Z}'_{\alpha}$, it follows that  $\mathcal{Z}'_{\alpha}$ does not meet $\theta'^{-1}_{\alpha}(E_{\alpha,m})$.
\item[2)]
Let $\mathfrak{V}_{\alpha}$ be the open subspace $f_{\alpha}(\mathfrak{U}_{\alpha})$, the image of $\mathfrak{U}_{\alpha}$ in $\mathfrak{Y}$. Replacing $\mathfrak{U}_{\alpha}$ by $\mathfrak{V}_{\alpha}$, $\mathcal{U}_{\alpha}$ by $\mathcal{V}_{\alpha}=\widetilde{(\mathfrak{V}_{\alpha})_{m_\alpha}}\times_{C_{m_{\alpha}}}\Spec(R_{\alpha})$, 
$\mathcal{U}'_{\alpha}$ by $\mathcal{V}'_{\alpha}=\widetilde{(\mathfrak{V}_{\alpha})_{m}}\times_{C_{m}}\Spec(R_{\alpha})$, and $Z_{\alpha}$ by $Z\cap \mathcal{V}_{\alpha}$, the conclusion of \textsc{Sub-step 1}
 implies that the same conclusion holds, i.e. the closure of the preimage of $Z\cap \mathcal{V}_\alpha$ in $\mathcal{V}'_{\alpha}$ does not meet the preimage of the singular locus of the special fiber of $\widetilde{(\mathfrak{V}_\alpha)_m}$ via the morphism
\[
\vartheta'_{\alpha}:
\mathcal{V}'_{\alpha}\rightarrow \widetilde{(\mathfrak{V}_\alpha)_m}.
\]
\item[3)]

Now we consider the following commutative diagram
\begin{equation}
	\xymatrix{
	 \mathcal{Y}_{m} \ar[dr]^{\eta_{\alpha}} \ar[dd] \ar[rr]^{\vartheta_{m}}  && \widetilde{\mathfrak{Y}_m} \ar'[d][dd]^{\tilde{\pi}_{m}} \ar[dr] \\
	 & \mathcal{Y}_{\alpha} \ar[dd] \ar[rr]^<<<<<<<<{\vartheta_{\alpha}} && \widetilde{\mathfrak{Y}_{m_{\alpha}}} \ar[dd]^{\tilde{\pi}_{m_{\alpha}}}  \\
	 \Spec(R_m) \ar[dr] \ar'[r][rr] && C_m \ar[dr] \\
	 & \Spec(R_{\alpha}) \ar[rr] && C_{m_{\alpha}}
	}
\end{equation}
where the front and the back square are cartesian. The morphism $\eta_{\alpha}$, restricted to $ (\mathcal{V}_{\alpha})_m=
\eta^{-1}_{\alpha}(\mathcal{V}_{\alpha})$, factors as the first row of the commutative diagram
\[
\xymatrix{
	\eta^{-1}(\mathcal{V}_{\alpha}) \ar[d] \ar[r] & \mathcal{V}'_{\alpha} \ar[r] \ar[d]  & \mathcal{V}_{\alpha}  \ar[d] \\
	\Spec(R_m) \ar[r] & \Spec(R'_{\alpha}) \ar[r] & \Spec(R_{\alpha}).
}
\]
The morphism $\vartheta_{m}$ restricted to $\eta^{-1}(\mathcal{V}_{\alpha})$ factors as
\[
\xymatrix{
	\eta^{-1}(\mathcal{V}_{\alpha}) \ar[d] \ar[r]^{\vartheta_m}  &  \widetilde{(\mathfrak{V}_{\alpha})_m} \\
	\mathcal{V}'_{\alpha}\ar[ur]_{\vartheta'_m} &.
}
\]
From the conclusion of \textsc{Sub-step 2}
, it follows that  the closure of $Z_m\cap \eta^{-1}(\mathcal{V}_\alpha)$ in $\eta^{-1}(\mathcal{V}_\alpha)$ does not meet $\vartheta_m^{-1}\big(D_m\cap \widetilde{(\mathfrak{V}_{\alpha})_m}\big)=\vartheta_m^{-1}\big(D_m)\cap \eta^{-1}(\mathcal{V}_{\alpha})$. Since $\eta^{-1}(\mathcal{V}_\alpha)$ is open in $\mathcal{Y}_m$, the closure $\mathcal{Z}_m$ of $Z_m$ in $\mathcal{Y}_m$ does not meet $\vartheta_m^{-1}\big(D_m)\cap \eta^{-1}(\mathcal{V}_{\alpha})$. As the assumption of \textsc{Step 1}, the union 
\[
\bigcup_{\alpha\in A}\eta^{-1}(\mathcal{V}_{\alpha})
\]
contains $\vartheta_m^{-1}\big(D_m)$. So we conclude that $\mathcal{Z}_m$ does not meet $\vartheta_m^{-1}\big(D_m)$, and the claim of \textsc{Step 1} is proved.
\end{enumerate}

\textsc{Step 2}: The statement is \emph{inverse étale local} on $\mathfrak{Y}$, in the following sense. Let 
\begin{equation}\label{graph-prop-prop-semisstablereduction-zero-dim-subschemes-1}
\xymatrix{
\mathfrak{U} \ar[dr]_{\pi'}  &&  \mathfrak{Y} \ar[dl]^{\pi} \ar[ll]_{f} \\
	& C  &
}	
\end{equation}
be a commutative diagram with $f$ an étale morphism.
Let $Z'=f(Z)$, with the reduced structure. Then set theoretically we have
\[
\mathcal{Z}_m\subset f^{-1}(\mathcal{Z'}_m).
\]
Consequently, if the statement holds for $\mathfrak{U}$, then it holds also for $\mathfrak{Y}$. 

\textsc{Step 3}: Combining \textsc{Step 1} and \textsc{Step 2}, 
the statement of the proposition reduces to the local model $\pi'$ in (\ref{graph-local-model-1}). This local model admits an obvious projective completion over $C=\mathbb{A}^1$ which is a bipartite simple degeneration. Thus  we can assume that $\mathfrak{Y}$ is projective over $C$.
We can deduce the proposition from Theorem \ref{thm-properness-X-projective}, as follows.

By Theorem \ref{thm-properness-X-projective}, $\mathcal{I}_{\mathfrak{Y}/C}^n$ is proper over $C$. Consider  the square on the right in the following diagram
\[
\xymatrix{
\Spec(K') \ar[r] \ar[d] &	\Spec(K) \ar[r]^>>>>>{f} \ar[d]^>>>>{i} & \mathcal{I}_{\mathfrak{Y}/C}^n \ar[d]^{p} \\
\Spec(R') \ar[r] \ar@{-->}[rru]^<<<<<<<<<{h}  &	\Spec(R) \ar[r]^>>>>>>{g} & C
}
\]
where $f$ is induced by the subscheme $Z$. By the valuative criterion for Deligne-Mumford stacks, there exists a finite extension $K'$ of $K$, and the normalization $R'$  of $R$ in $K'$ such that there exists a dashed filling $h$ in the diagram. We replace $R'$ by the localization of $R'$ at a closed point, so that $R'$ is a DVR. By definition, such a morphism $h$ represents a $\mathbb{G}_m^n$-torsor (in the étale topology) $P$ and a $\mathbb{G}_m^n$-equivariant morphism $P\rightarrow \mathbf{H}^n_{\st}(\mathfrak{Y}/C)$. There exists an étale extension $R'\rightarrow R''$ such that the torsor $P$ acquires a section over $\Spec(R'')$. Again we can assume that $R''$ is a DVR. Thus there is a morphism $\tilde{h}:\Spec(R'')\rightarrow \mathbf{H}^n_{\st}(\mathfrak{Y}/C)$ rendering the following diagram commutative:
\[
\xymatrix{
	\Spec(R'') \ar[r]^<<<<<{\tilde{h}} \ar[d] & \mathbf{H}^n_{\st}(\mathfrak{Y}/C) \ar[d] & \\
	\Spec(R') \ar[r]^{h} & \mathcal{I}_{\mathfrak{Y}/C}^n \ar[r] & C[n]=C\times_{\mathbb{A}^1}\mathbb{A}^{n+1}
}
\]
Denote the composition $\Spec(R'')\rightarrow C[n]$ by $q$. Let $\varpi''$ be a uniformizer of $R''$, and suppose $q^* t_i=\xi_i \varpi''^{e_i}$,  for $1\leq i\leq n+1$, where $\xi_i$ are units of $R''$. Then $\mathfrak{Y}[n]\times_{C[n]}\Spec(R'')$ is étale locally the subscheme of 
\[
\Spec\ R''[x,y,\dots]\times \prod_{i=1}^{n}\mathrm{Proj}\ \Bbbk[u_i,v_i]
\]
defined by (the units $\xi_i$ are absorbed into $u_i$ or $y$)
\begin{eqnarray}\label{eq-local-equations-4}
	\begin{cases}
	u_1 x=v_1 \varpi''^{e_1},\\
	u_i v_{i-1}=v_i u_{i-1} \varpi''^{e_i},\ 1<i\leq n,\\
	v_{n} y=u_{n} \varpi''^{e_{n+1}}.
	\end{cases}
\end{eqnarray}
If some $e_i=0$, then deleting this index will yield an isomorphic scheme. So we can assume $e_i\geq 1$ for $1\leq i\leq n$.
The morphism $\tilde{h}$ represents  a closed subscheme $W$ of 
$\mathfrak{Y}[n]\times_{C[n]}\Spec(R'')$ which is finite and flat of degree $n$ over $\Spec(R'')$, such that the support of the central fiber of  $W$ satisfies the condition in Proposition \ref{thm-stable-locus}, and the generic fiber of $W$ is isomorphic to $Z$. In particular $q^{-1}W$ does not meet the intersections of any two components of the special fiber of $\mathfrak{Y}[n]\times_{C[n]}\Spec(R'')$. Further blowing up successively the strict transform of $V(u_i,\varpi^k)$ and $V(y,\varpi^k)$ for appropriate $i$ and $k$'s, we complete  the Construction \ref{cons-specialization-0dim-subscheme-3} for $R''$, and the resulted subscheme $W''$ satisfies the requirement of our statement. Hence the proof is complete.
\end{proof}

\begin{theorem}\label{thm-properness-X-algebraic-space}
Let $\mathfrak{X}$ be a smooth algebraic space over $\Bbbk$, and $\pi:\mathfrak{X}\rightarrow C$ a bipartite simple  degeneration. Then $\mathcal{I}_{\mathfrak{X}/C}^n$ is a Deligne-Mumford stack proper over $C$.
\end{theorem}
\begin{proof}
To show that the quotient stack $\mathcal{I}_{\mathfrak{X}/C}^n=[\mathbf{H}^n_{\st}(\mathfrak{X}/C)/G[n]]$ is a DM-stack, it suffices to show that for every algebraically closed field $k$, and $z\in \mathcal{I}_{\mathfrak{X}/C}^n(k)$, the automorphism group scheme $\underline{\mathrm{Aut}}_z$ is a reduced finite $k$-scheme.   By Definition \ref{def-good-degeneration-Hilbert-scheme-points}, when $z$ represents a non-empty 0-dimensional subspace (denoted still by $z$) of a $\mathbb{P}^1$-bundle $P$ over an algebraic space over $k$, the stabilizer subgroup scheme of $\mathbb{G}_m$ is a reduced finite $k$-scheme. Since $\underline{\mathrm{Aut}}_z$ is a subscheme of $\underline{\mathrm{Aut}}_{z_{\mathrm{red}}}$, we can assume that $z$ is a finite set of reduced closed points. Taking a $\mathbb{P}^1$ fiber containing a point in $z$, it suffices to consider the case $P=\mathbb{P}^1$, and in this case $\underline{\mathrm{Aut}}_{z_{\mathrm{red}}}$ is just the group scheme associated to the abstract finite subgroup of $k^*$ preserving the support of $z$.  So $\mathcal{I}_{\mathfrak{X}/C}^n$ is a DM-stack.

Since  $\mathbf{H}^n_{\st}(\mathfrak{X}/C)$ is separated, $\mathcal{I}_{\mathfrak{X}/C}^n=[\mathbf{H}^n_{\st}(\mathfrak{X}/C)/G[n]]$ is separated. We need only to show the existence part of the valuative criterion for the properness of $\mathcal{I}_{\mathfrak{X}/C}^n\rightarrow C$. We need to consider the commutative diagram of the form
\[
\xymatrix{
	\Spec(K) \ar[r]^>>>>>{f} \ar[d] & \mathcal{I}_{\mathfrak{X}/C}^n \ar[d]^{p} \\
	\Spec(R) \ar[r]^>>>>>>{g} & C
}
\]
where $R$ is a DVR and $K$ the fraction field of $R$. By Proposition \ref{prop-flatness-good-degeneration}, $p$ is flat. So $p^{-1}(C\setminus\{0\})$ is open and dense in $\mathcal{I}_{\mathfrak{X}/C}^n$. Thus from (\cite[Theorem 4.19]{DeM69}, \cite[Tag 0CM5]{StPr}), without loss of generality, we can assume $\Spec(K)$ maps into $p^{-1}(C\setminus\{0\})$.  The following proof is basically a reverse of \textsc{Step 3} of the proof of Proposition \ref{prop-semisstablereduction-zero-dim-subschemes}. 

  Let $\varpi$ be a uniformizer of $R$. Suppose $g^*(xy)=\xi \varpi^m$, where $\xi\in R$ is a unit. Blowing up $\mathfrak{X}\times_C \Spec(R)$ successively along $(y,\varpi^{m-1})$,\dots,$(y,\varpi)$ we obtain a regular scheme $Y$ as in the setup of the construction \ref{cons-specialization-0dim-subscheme-1}. The morphism $f$ induces a $0$-dimensional closed subscheme $Z$ of $Y_K$. 
Then we apply Proposition \ref{prop-semisstablereduction-zero-dim-subschemes} to $Y$ and $Z$, thus obtain  $s:\Spec(R')\rightarrow \Spec(R)$, $Y''$ over $\Spec(R')$, and $W\subset Y''$.  Denote $g'=g\circ s$. Suppose $g'(xy)=\xi' \varpi'^{k}$, where $\xi'$ is a unit of $R'$. Replacing $R'$ be a further ramified extension if necessary, we can assume $k\geq n+1$. 
Then $Y''$ is isomorphic to the result of the successive blowing  up of $\mathfrak{X}\times_C \Spec(R')$ along $V(y,\varpi'^{k-1})$, the strict transform of  $V(y,\varpi'^{k-2})$, \dots, the strict transform of  $V(y,\varpi')$. The special fiber of $Y''$ has $k+1$ components $P_0=Y_1$, $P_1$, \dots, $P_{k-1}$, $P_k=Y_2$, arranged in the natural order such that $P_i\cap P_{i+1}$ has non-empty intersections. Suppose 
\[
l_i=\mathrm{length}(W\cap P_i).
\]
 Let $l_{i_1},\dots,l_{i_b}$ be the nonzero elements in $(l_1,\dots,l_{k-1})$, and $1\leq i_1<\dots<i_b\leq k-1$. The properness of $\mathfrak{X}\rightarrow C$ implies that the special fiber of $W$ has a length equal to $n$. Thus $l_0+l_k+\sum_{j=1}^b l_{i_j}=n$, and $b\leq n$.
We lift the morphism $g'$ to $q:\Spec(R')\rightarrow C[b]=C\times_{\mathbb{A}^1}\mathbb{A}^{b+1}$ 
\begin{equation}\label{diag-proof-thm-properness-X-algebraic-space-1}
 \xymatrix{
\mathbf{H}^{b}_{\st}(\mathfrak{X}/C) \ar[r]	 & C[b] \ar[d] \\
 \Spec(R') \ar@{-->}[u]^{\tilde{h}} \ar[ru]^{q} \ar[r]^>>>>>>{g'} & C
 }
\end{equation} 
 by requiring $q^*(t_{j})=\varpi^{l_{i_j}}$ for $1\leq j\leq b$, and $q^*(t_{b+1})=\varpi^{k-\sum_{j=1}^b l_{i_j}}$.
Then $W$ induces a morphism $\tilde{h}:\Spec(R')\rightarrow \mathbf{H}^b_{\st}(\mathfrak{X}/C)$ rendering the diagram (\ref{diag-proof-thm-properness-X-algebraic-space-1}) commutative.
Let $\iota$ be the imbedding $C[b]\hookrightarrow C[n]$ defined by
\[
\begin{cases}
\iota^*(t_{l_0+\sum_{j=1}^e l_{i_j}})=t_e,\ \mbox{for}\ 1\leq e\leq b\\
\iota^*(t_{n+1})=t_{b+1}\\
\iota^*(t_i)=1,\ \mbox{for}\ i\not\in\{l_0+l_{i_1},l_0+l_{i_1}+l_{i_2},\dots,\sum_{j=0}^{b}l_{i_j},n+1\}
\end{cases}.
\]
Then $\iota\circ q$ and $W$ induce a morphism $\Spec(R')\rightarrow \mathbf{H}^{n}_{\st}(\mathfrak{X}/C)$ rendering the diagram (\ref{diag-proof-thm-properness-X-algebraic-space-1}), replacing $b$ by $n$, commutative. The induced morphism $h$ in the following diagram
\[
\xymatrix{
\Spec(K') \ar[r] \ar[d]  & \mathcal{I}_{\mathfrak{X}/C}^n \ar[d]^{p} \\
\Spec(R') \ar[r] \ar@{-->}[ru]^<<<<<<<<<{h}  & C
}
\]
makes diagram commute. 
 This verifies the existence part of the valuative criterion, and the properness of $\mathcal{I}_{\mathfrak{X}/C}^n$ over $C$ follows.
\end{proof}

\begin{remark}
One can regard the statement of Proposition 
\ref{prop-semisstablereduction-zero-dim-subschemes} as a kind of  \emph{semistable reduction of zero dimensional subschemes}, which in the case of relative dimension 1 is very close to the  reduction of semistable curves with marked points. 
The above proofs in fact show that it is equivalent to  the properness of $\mathcal{I}_{\mathfrak{X}/C}^n$.
\end{remark}

\section{Flatness of the good degenerations of Hilbert spaces of points}\label{sec:flatness}

\subsection{Flatness of Hilbert spaces of points}
\begin{theorem}\label{thm-flatness-Hilb}
Let $\Bbbk$ be a field, $S$ a scheme of finite type over $\Bbbk$.  Let $X$ be an algebraic space separated over $\Bbbk$, and  $\pi:X\rightarrow S$ a smooth  morphism. Let $n$ be a positive integer. Then the relative Hilbert space  $\mathrm{Hilb}^{n}(X/S)$ is flat over $S$.
\end{theorem}

The intuition is that a smooth morphism is locally trivial  in the analytic topology if we are working over $\mathbb{C}$, and there is no obstruction for a zero dimensional closed subscheme on a fiber to move in the horizontal direction in an analytically trivial fibration. One difficulty in giving an algebraic proof (even for schemes)  is that for an étale morphism $U\rightarrow V$ of schemes over $S$, there is no natural morphism from $\mathrm{Hilb}^n(U/S)$ to $\mathrm{Hilb}^n(V/S)$. To solve this we make use of the following construction of Rydh \cite{Ryd11}, which originates from M. Artin \cite[Appendix]{Art74}.

For an algebraic space $X$ over $S$, let $\mathscr{H}_{X/S}^n$ be the category whose objects are pairs of morphisms of $S$-spaces $(p,q)$ in a diagram
\[
\xymatrix{
	Z\ar[d]_p \ar[r]^{q} & X\\
	T
}
\]
where $p$ is flat and finite of rank $n$. A morphism in $\mathscr{H}_{X/S}^n$ is a pair $(\varphi,\psi)$ in the commutative diagram
\[
\xymatrix{
	Z_1\ar[d]_{p_1} \ar[r]^{\varphi} \ar@/^1pc/[rr]^{q_1} & Z_2 \ar[d]_{p_2} \ar[r]^{q_2} & X\\
	T_1 \ar[r]^{\psi} & T_2
}
\]
such that the square is cartesian. By \cite[Theorem 4.4]{Ryd11} $\mathscr{H}_{\mathscr{X}/S}^n$ is an algebraic stack over $S$.  For a morphism $f:X\rightarrow Y$ of $S$-spaces, there is an obvious natural morphism $f_*:\mathscr{H}_{X/S}^n\rightarrow \mathscr{H}_{Y/S}^n$. Moreover, let $\mathrm{Hilb}^n_{X\rightarrow Y}$ be the subfunctor of $\mathrm{Hilb}^n(X/S)$ parametrizing families 
$Z\hookrightarrow X\times_S T$ such that the composition 
$Z\hookrightarrow X\times_S T\xrightarrow{f_T} Y\times_S T $ 
is a closed immersion. Then the following commutative diagram
\[
\xymatrix{
	\mathrm{Hilb}^n_{X\rightarrow Y} \ar[r]^{f_*} \ar[d] & \mathrm{Hilb}^n(Y/S) \ar[d] \\
	\mathscr{H}_{X/S}^n \ar[r]^{f_*} & \mathscr{H}_{Y/S}^n
}
\]
 is 2-cartesian.
\begin{proof}[Proof of Theorem \ref{thm-flatness-Hilb} when $X$ is a scheme]
Let  $U$ be an open subset of $X$ such that there is an étale $S$-morphism  $f:U\rightarrow \mathbb{A}^d_S$ for a certain integer $d$. The natural morphism $\mathrm{Hilb}^n(U/S)\rightarrow \mathscr{H}_{U/S}^n$ is an open immersion (\cite[Proposition 1.9]{Ryd11}), and $\mathscr{H}_{U/S}^n \rightarrow \mathscr{H}_{\mathbb{A}^d_S/S}^n$ is étale (\cite[Theorem 3.11]{Ryd11}). But by definition $\mathscr{H}_{\mathbb{A}^d_S/S}^n\cong \mathscr{H}_{\mathbb{A}^d_\Bbbk/\Bbbk}^n\times_\Bbbk S$. It follows that $\mathrm{Hilb}^n(U/S)$ is flat over $S$.

Let $z$ be a point of $\mathrm{Hilb}^n(X/S)$. It represents a point $s\in S$ and a length $n$ closed subscheme $Z$ of $X_s$. First we consider the case that $Z$ is supported at a single closed point $x$ of $X_s$. Since $X$ is smooth over $S$, there exists an open subset $U$ containing $x$ and an étale morphism $f:U\rightarrow \mathbb{A}^d_S$ for some $d$. Then $z$ lies in the image of the open immersion $\mathrm{Hilb}^n(U/S)\rightarrow\mathrm{Hilb}^n(X/S)$, so by the first paragraph $\mathrm{Hilb}^n(X/S)$ is flat over $S$ in an open neighborhood of  $z$.

Now consider the general case and suppose that the support of $z$ is $\{x_1,\dots,x_m\}$. Take open subsets $V_1,\dots,V_m$ of $X$ such that $x_i\in V_i$, and $x_j\not\in V_i$ for $1\leq i,j\leq m$ and $i\neq j$, and each $V_i$ admits an étale $S$-morphism  $V_i\rightarrow \mathbb{A}^d_S$.  Denote by $\iota_i$ the open immersion $V_i\hookrightarrow X$. Let $V=\bigsqcup_{i=1}^{m}V_i$. The collection of the open immersions $\iota_i$ gives a morphism $\iota:V\rightarrow X$, which is étale. In the cartesian diagram
\[
\xymatrix{
	\mathrm{Hilb}^n_{V\rightarrow X} \ar[r]^{\iota_*} \ar[d] & \mathrm{Hilb}^n(X/S) \ar[d] \\
	\mathscr{H}_{V/S}^n \ar[r]^{\iota_*} & \mathscr{H}_{X/S}^n
}
\]
the bottom $\iota_*$ is étale again by \cite[Theorem 3.11]{Ryd11}, so the upper $\iota_*$ is étale. Let $Z_i=\iota_i^{-1}(Z)$, and denote the length of $Z_i$ by $n_i$, and denote the corresponding point of $\mathrm{Hilb}^{n_i}(V_i/S)$ by $z_i$. Let $Z'=\bigsqcup_{i=1}^{m}Z_i$ be a closed subscheme of $V_s$ of length $n$, and denote the corresponding point of $\mathrm{Hilb}^n(V/S)$ by $z'$.  Then $\iota_* Z'=Z$. So to show the flatness of $\mathrm{Hilb}^n(X/S)$ at $z$, it suffices to show the flatness of $\mathrm{Hilb}^n_{V\rightarrow X}$ at $z'$. Since $\mathrm{Hilb}^n_{V\rightarrow X}$ is an open subscheme of $\mathrm{Hilb}^n(V/S)$, we are in turn left to show that $\mathrm{Hilb}^n(V/S)$ is flat over $S$ at $z'$. But there is an open immersion 
\[
\prod_{i=1}^m \mathrm{Hilb}^{n_i}(V_i/S)\rightarrow \mathrm{Hilb}^n(V/S)
\]
which maps $(z_1,\dots,z_m)$ to $z'$. So the question is reduced to the flatness of $\mathrm{Hilb}^{n_i}(V_i/S)$ at $z_i$, for $1\leq i\leq m$. This is the single-point-support case that has been treated above. Hence the proof is completed. 
\end{proof}

\begin{lemma}\label{lem-zerodim-subschme-etale-chart}
Let $S$ be a scheme. Let $X$ be a quasi-separated algebraic space locally of finite presentation over $S$. Let $s$ be a point of $S$, and $Z$ be a $0$-dimensional subspace of $X_s$ (the fiber of $X$ at $s$) supported at one point $x$ of $X_s$. Then there exists an étale morphism $f:Y\rightarrow X$ where $Y$ is a scheme, and a $0$-dimensional closed subscheme $W$ of $Y_s$, such that $Z=f_* W$, i.e. the composition $W\hookrightarrow Y_s\rightarrow X_s$ is a closed immersion and isomorphic to $Z\hookrightarrow X_s$.
\end{lemma}
\begin{proof}By \cite[Theorem 19.1]{AHD19}, there exists an affine scheme $U$, a point $u$ of $U$, and an étale morphism $f:(U,u)\rightarrow (X,x)$, which induces an isomorphism of residue fields. We delete the pre-images of $x$ in $U$ other than $u$, and denote the resulted scheme by $Y$, and the morphism $g:Y\rightarrow X$. Then $Y_s\times_{X_s} Y_s\rightrightarrows Y_s$ is a presentation of an open subspace of $X_s$ containing $x$, and 
$g_s^{-1}(x)$ consists of a unique point $u$. So $\mathrm{pr}_1^{-1}(u)=\mathrm{pr}_1^{-1}(u)=\{(u,u)\}$, and the induced homomorphisms of residue fields
\[
\kappa(u)\xrightarrow{\mathrm{pr}_1^{*}}\kappa((u,u))\xleftarrow{\mathrm{pr}_2^{*}} \kappa(u)
\]
 are isomorphisms. Consequently, they induce isomorphisms of the henselizations
 \[
 \mathcal{O}_{Y_s,u}^{h}\xrightarrow{\mathrm{pr}_1^{*}} 
 \mathcal{O}^h_{Y_s\times_{X_s} Y_s,(u,u)}\xleftarrow{\mathrm{pr}_2^{*}}
 \mathcal{O}_{Y_s,u}^{h},
 \]
 and thus isomorphisms of the completions
 \[
 \widehat{\mathcal{O}}_{Y_s,u}\xrightarrow{\mathrm{pr}_1^{*}} 
 \widehat{\mathcal{O}}_{Y_s\times_{X_s} Y_s,(u,u)}\xleftarrow{\mathrm{pr}_2^{*}}
 \widehat{\mathcal{O}}_{Y_s,u}.
 \]
 By descent theory, a closed subspace of $X_s$ supported at $x$ corresponds to subschemes $Z_i$ of $Y_s$ supported at $u$, for $i=1,2$, satisfying 
$\mathrm{pr}_1^{-1}(Z_1)=\mathrm{pr}_2^{-1}(Z_2)$. But a subscheme of $Y_s$ supported at $u$ is a closed subscheme of $\Spec(\widehat{O}_{Y_s,u})$ of finite length. So the above isomorphism of formal completions yields that a subspace $Z$ of $X_s$ supported at $x$ uniquely corresponds to a subscheme $W$ of $Y_s$ supported at $u$, and $f_* W=Z$.
\end{proof}

\begin{proof}[Proof of Theorem \ref{thm-flatness-Hilb} when $X$ is an algebraic space]
Let $z$ be a point of $\mathrm{Hilb}^n(X/S)$. It represents a point $s\in S$ and a length $n$ closed subspace $Z$ of $X_s$. First suppose that $Z$ is supported at a single closed point $x$ of $X_s$. By Lemma \ref{lem-zerodim-subschme-etale-chart}, there is an étale morphism $f:U\rightarrow X$ where $U$ is a scheme, such that in the following cartesian diagram
\[
\xymatrix{
	\mathrm{Hilb}^n_{U\rightarrow X} \ar[r]^{f_*} \ar[d] & \mathrm{Hilb}^n(X/S) \ar[d] \\
	\mathscr{H}_{U/S}^n \ar[r]^{f_*} & \mathscr{H}_{X/S}^n
}
\]
$z$ lies the image of the upper $f_*$. The argument of the proof for the scheme case, and that $\mathrm{Hilb}^n(U/S)$ is smooth over $S$ as we have shown,  yields that $\mathrm{Hilb}^n(X/S)$ is flat near $z$.

Now suppose that the support of $z$ is $\{x_1,\dots,x_m\}$. Denote the subspace of $Z$ supported at $x_i$ by $Z_i$ and the corresponding point by $z_i \in \mathrm{Hilb}^{n_i}(X/S)$. By Lemma \ref{lem-zerodim-subschme-etale-chart} there exist schemes $V_1,\dots,V_m$ and étale morphisms $f_i:V_i\rightarrow X$, and $w_i\in \mathrm{Hilb}^{n_i}(V_i/S)$ such that $f_{i*}(w_i)=z_i$. Then the argument of the last paragraph of the scheme case completes the proof.
\end{proof}

\subsection{Flatness of good degenerations}
In this subsection, let $C$ be a smooth affine curve over a field $\Bbbk$, $\mathfrak{X}$ be a smooth algebraic space over $\Bbbk$, and $\pi:\mathfrak{X}\rightarrow C$ be a bipartite simple  degeneration. Let $n$ be a natural number.

\begin{lemma}\label{lemma-fiber-smoothness}
Let  $\mathbf{x}$ be a point of $\mathfrak{X}[n]$. Consider the coordinates $(x,y,u_1,\dots,u_{n},v_{1},\dots,v_n,t_1,\dots,t_{n+1})$ in a local model near $\mathbf{x}$.
\begin{enumerate}
	\item[(i)] The fiber $\pi^{-1}\pi(\mathbf{x})$ is singular at $x$ iff $x=u_1=0$ at $\mathbf{x}$, or $y=v_n=0$ at $\mathbf{x}$, or there exists $k$, $2\leq k\leq n$, such that $u_k=v_{k-1}=0$;
	\item[(ii)] The fiber $\pi^{-1}\pi(\mathbf{x})$ is smooth at $\mathbf{x}$ iff $x\neq 0$ at $\mathbf{x}$, or $y\neq 0$ at $\mathbf{x}$, or there exists $k$, $1\leq k\leq n$, such that $u_k\neq 0$ and $v_k\neq 0$ at $\mathbf{x}$.
\end{enumerate}
\end{lemma}
\begin{proof}
This follows directly from the defining equations (\ref{eq-local-equations}) of the local model.
\end{proof}

\begin{lemma}\label{lemma-local-smoothness}
Let $\mathbf{x}$ be a  point of $\mathfrak{X}[n]$. Suppose that the fiber $\pi^{-1}\pi(\mathbf{x})$ is smooth at $\mathbf{x}$. Then $\pi[n]$ is smooth in an open neighborhood of $\mathbf{x}$. 
\end{lemma}
\begin{proof}
Each fiber of $\mathfrak{X}[n]\rightarrow C[n]$ is a simple normal crossing scheme. Let $\mathbf{x}$ be a point such that the fiber $\pi^{-1}\pi(\mathbf{x})$ is smooth at $\mathbf{x}$. Consider a local chart of the form (\ref{eq-local-chart}) containing $\mathbf{x}$.  Then by Lemma \ref{lemma-fiber-smoothness} at least one of the following three cases happens:
\begin{enumerate}
 	\item[(i)]$x\neq 0$ at $\mathbf{x}$;
 	\item[(ii)] $y\neq$ at $\mathbf{x}$;
 	\item[(iii)] there exists $k$, $1\leq k\leq n$, such that neither of $u_{k}$ and $v_{k}$ vanishes at $\mathbf{x}$.
 \end{enumerate}   
 Assume case (iii). Then the equations (\ref{eq-local-equations}) imply that none of   $u_1,\dots,u_{k}$ and $v_{k},\dots,v_n$ vanishes at $\mathbf{x}$. Consider the open subset $U_k$ defined by the nonvanishing of these coordinates. On $U_k$ the system of equations (\ref{eq-local-equations}) are solved as
\begin{eqnarray*}
&&x=t_1\cdots t_{k}\cdot \frac{v_{k}}{u_{k}},\\
&& \frac{v_i}{u_i}=t_{i+1}\cdots t_{k}\cdot \frac{v_{k}}{u_{k}},\ 1\leq i\leq k-1,\\
&& \frac{u_i}{v_i}=\frac{u_{k}}{v_{k}}\cdot t_{k+1}\cdots t_{i},\ k+1\leq i\leq n,\\
&&y=\frac{u_{k}}{v_{k}}\cdot t_{k+1}\cdots t_{n+1}.
\end{eqnarray*}
Then the restriction of $\pi[n]$ to $U_k$ is isomorphic to the projection $\mathbb{G}_m\times C[n]\xrightarrow{\mathrm{pr}_2}C[n]$, $(\frac{u_k}{v_k},t_1,\dots,t_{n+1})\mapsto (t_1,\dots,t_{n+1})$, so the conclusion is true in this case. The cases (i) and (ii) are similar.
\end{proof}

\begin{proposition}\label{prop-flatness-good-degeneration}
The Deligne-Mumford stack $\mathcal{I}^n_{\mathfrak{X}/C}$ is flat over $C$.
\end{proposition}
\begin{proof} By definition, $\mathcal{I}^n_{\mathfrak{X}/C}=[\mathbf{H}^n_{\st}/G[n]]$, and $\mathbf{H}^n_{\st}(\mathfrak{X}/C)$ is contained in $\mathbf{H}^n_{sm}(\mathfrak{X}/C)$. By Lemma \ref{lemma-local-smoothness} and Theorem \ref{thm-flatness-Hilb}, $\mathbf{H}^n_{sm}(\mathfrak{X}/C)$ is flat over $C[n]$,  thus  $\mathbf{H}^n_{\st}(\mathfrak{X}/C)$ is also flat over $C[n]$. In the commutative diagram
\[
\xymatrix{
	\mathbf{H}^n_{\st} \ar[r]^>>>>>>>>>>{q_1} \ar[d]_{p_2} & C[n] \ar[d]_{p_1} \ar[dr]^{p_4} &  \\
	[\mathbf{H}^n_{\st}/G[n]] \ar[r]^>>>>>{q_2} & [C[n]/G[n]] \ar[r]^>>>>>{p_3} & C
}
\]
both $p_1$ and $p_2$ are faithfully flat. 
We have shown the flatness of $q_1$. So $q_2$ is flat. Since  $p_4$ is flat, $p_3$ is also flat. Hence the composition $p=p_3\circ q_2$ is flat.
\end{proof}

\section{Decomposition of the central fiber of a good degeneration}\label{sec:decompositionCentralFiber}

The goal of this section is Proposition \ref{prop-central-fiber-relativeHilbertscheme}, which provides a \emph{scheme theoretical} decomposition of the central fiber of $\mathcal{I}^n_{\mathfrak{X}/C}$.
\begin{definition}
Let $(Y,D)$ be a pair of smooth algebraic spaces over $\Bbbk$, such that $D$ is a divisor of $Y$.  Regarding $D$ as a divisor of the fiber over $0$ of the trivial family $Y
\times \mathbb{P}^1$, we set $\mathfrak{S}_{Y,D}=\mathrm{Bl}_{D}(Y\times \mathbb{P}^1)$, the blow-up of $Y\times \mathbb{P}^1$ along $D$. Then $\mathfrak{S}_{Y,D}$ is  a bipartite simple  degeneration of $Y$, whose central fiber is $Y\cup_{D} \mathbb{P}(N_{D/Y}\oplus\mathcal{O}_D)$. We call $\mathfrak{S}_{Y,D}$ the \emph{standard degeneration} associated with $(Y,D)$.
\end{definition}

\begin{definition}\label{def-expanded-degeneration-right-pair}
Let $(Y,D)$ be a pair of smooth proper algebraic spaces over $\Bbbk$, such that $D$ is a divisor of $Y$. Let $\mathbb{G}_m^n$ act on $\mathbb{A}^n=\Spec\ \Bbbk[t_1,\dots,t_n]$ as $(\sigma_1,\dots,\sigma_n).(t_1,\dots,t_n)=(\sigma_1 t_1,\dots,\sigma_n t_n)$. For $n\geq 0$,  we define an algebraic space $(D|Y)^{[n]}$ with a $\mathbb{G}_m^n$-action, together with an equivariant birational morphism $\rho_{D|Y}^{[n]}:(D|Y)^{[n]}\rightarrow Y\times \mathbb{A}^n$, recursively as follows.
\begin{enumerate}
	\item[(i)] $(D|Y)^{[0]}=Y$;
	\item[(ii)] Define $(D|Y)^{[1]}$ as the blow-up of $Y\times \mathbb{A}^1$ along $D\times V(t_1)$. The $\mathbb{G}_m$ action on $(D|Y)^{[1]}$ is the  action induced by its action on $\mathbb{A}^1$ and the blowing-up, whose center is $\mathbb{G}_m$-invariant; 
	\item[(iii)] Suppose that $\rho_{D|Y}^{[n-1]}:(D|Y)^{[n-1]}\rightarrow Y\times \mathbb{A}^{n-1}$ has been constructed. Consider the fiber product $(D|Y)^{[n-1]}\times_{t_{n-1},\mathbb{A}^{1},m}\mathbb{A}^2$, where $t_{n-1}$ in the subscript means the composition 
	\[
	(D|Y)^{[n-1]}\xrightarrow{\rho_{D|Y}^{[n-1]}}Y\times \mathbb{A}^{n-1}\xrightarrow{\mathrm{pr}_2} \mathbb{A}^{n-1}\xrightarrow{\mathrm{pr}_{n-1}}\mathbb{A}^1,
	\]
	and $m$ is the multiplication $\mathbb{A}^2\rightarrow \mathbb{A}^1$. There is an obvious morphism 
	\[
	 (D|Y)^{[n-1]}\times_{t_{n-1},\mathbb{A}^{1},m}\mathbb{A}^2\rightarrow Y\times \mathbb{A}^n,
	 \] 
	 \[
	 (z)\times (t,t')\mapsto \big(\mathrm{pr}_1\circ \rho_{D|Y}^{[n-1]}(z)\big)\times \big(\mathrm{pr}_{\{1,2,\dots,n-2\}}\circ \mathrm{pr}_2\circ \rho_{D|Y}^{[n-1]}(z),t,t'\big),
	 \]
	 where  $z\in (D|Y)^{[n-1]}$, and
	 \[
	 \mathrm{pr}_{\{1,2,\dots,n-2\}}: \mathbb{A}^{n-1}\rightarrow \mathbb{A}^{n-2}
	 \]
	 is the projection to the first $n-2$ factors. 
	 Define $(D|Y)^{[n]}$ as the blow-up of $(D|Y)^{[n-1]}\times_{t_{n-1},\mathbb{A}^{1},m}\mathbb{A}^2$ along
	 the strict transform of $D\times V(t_n)$. The $\mathbb{G}_m^n$-action on $(D|Y)^{[n]}$ is  induced by its action on $\mathbb{A}^n$.
\end{enumerate}
\end{definition}

\begin{definition}\label{def-expanded-degeneration-left-pair}
Let $(Y,D)$ be a pair of smooth proper algebraic spaces over $\Bbbk$, such that $D$ is a divisor of $Y$. We define $(Y|D)^{[n]}$, as an algebraic space, equal to $(D|Y)^{[n]}$, but with a $\mathbb{G}_m^n$-action induced by the action $\mathbb{G}_m^n\times \mathbb{A}^n\rightarrow \mathbb{A}^n$ given as $(\sigma_1,\dots,\sigma_n).(t_1,\dots,t_n)=(\sigma_n t_1,\dots,\sigma_1 t_n)$.
\end{definition}

\begin{remark}\label{rmk-relative-Hilb-space-local-model}
An explicit definition of $(D|Y)^{[n]}$ via the local model is as follows. Let $\Spec\ \Bbbk[y,\dots]$ be a local model of $Y$ such that $D$ is defined by $y=0$.
We define $(D|Y)^{[n]}$ in $n$ steps:
\begin{enumerate}
	\item[(1)] Blow up $\Spec\ \Bbbk[x,\dots]\times \mathbb{A}^n$ along the subscheme  $y=t_n=0$. The resulting scheme $U_1$ is the close subscheme of 
	\[
	\Spec\ \Bbbk[y,\dots] \times \mathbb{A}^n\times \mathbb{P}^1=\Spec\ \Bbbk[y,\dots] \times \Spec\ \Bbbk[t_1,\dots,t_n]\times \mathrm{Proj}\ \Bbbk[u_n,v_n]
	\]
	defined by 
	\[
	v_n y=u_n t_{n}.
	\]
	\item[(2)] Blow up $U_1$ along the  subscheme $u_n=t_{n-1}=0$. Since $v_n\neq 0$ in an open neighborhood of this subscheme, and in this neighborhood $u_n/v_n$ is a coordinate, the resulting scheme $U_2$ is the close subscheme of 
	\[
	\Spec\ \Bbbk[y,\dots] \times \mathbb{A}^n\times \mathbb{P}^2=
	\Spec\ \Bbbk[y,\dots] \times \Spec\ \Bbbk[t_1,\dots,t_n]\times \mathrm{Proj}\ \Bbbk[u_{n-1},v_{n-1}]\times \mathrm{Proj}\ \Bbbk[u_n,v_n]
	\]
	defined by 
	\[
	\begin{cases}
	u_n v_{n-1}=v_{n}u_{n-1}t_{n-1},\\
	v_n y=u_n t_{n}.
	\end{cases}
	\]
\end{enumerate}
Repeating this process, in the $i$-th step we blow up the subscheme $u_{n+2-i}=t_{n+1-i}=0$. Finally we obtain a closed subscheme $U_n$ of 
	\[
	 	\Spec\ \Bbbk[y,\dots] \times \mathbb{A}^{n}\times \mathbb{P}^{n}
	 \] 
	 defined by 
	 \[
	 \begin{cases}
	 u_{i+1} v_{i}=v_{i+1} u_{i} t_{i},& 1\leq i\leq n-1,\\
	 v_n y=u_n t_n.& 
	 \end{cases}
	 \]
This in fact well defines a scheme $(D|Y)^{[n]}$ whose local model is $U_n$, for in each step the center of the blow-up is independent of the choice of the local coordinates of $(D,Y)$ and the defining equation of $D$. One can check that the two definitions coincide, by explicitly write down the local equations in  Definition \ref{def-expanded-degeneration-right-pair}. 
\pqed
\end{remark}

\begin{definition}\label{def-expanded-degeneration-middle-pair}
Let $D$ be a smooth proper algebraic space over $\Bbbk$, $N$ a line bundle on $D$. Let $\mathbb{P}_{D,N}=\mathbb{P}(N\oplus \mathcal{O}_D)$. We identify the $0$-section of $N$ with a section of $\mathbb{P}_{D,N}$, and denote it by $D_0$. Then $(\mathbb{P}_{D,N}, D_0)$ is a pair of smooth algebraic spaces over $\Bbbk$ and $D_0$ is a divisor of $\mathbb{P}_{D,N}$. 
 We define $(D,N)^{[n]}$ to be  $(D_0|\mathbb{P}_{D,N})^{[n-1]}$.  There is a $\mathbb{G}_m^{n-1}$-action on $(D_0|\mathbb{P}_{D,N})^{[n-1]}$. Moreover the $\mathbb{G}_m$-action on $N$, $\sigma.v=\sigma v$, induces a $\mathbb{G}_m$-action on $\mathbb{P}_{D,N}$, and thus on $(D,N)^{[n]}$. We let $\mathbb{G}_m^n=\mathbb{G}_m\times \mathbb{G}_m^{n-1}$ act on $(D,N)^{[n]}$ via the product of the two actions.
\end{definition}

Let $\mathfrak{X}$ be a smooth algebraic space over $\Bbbk$, and $\pi:\mathfrak{X}\rightarrow C$ be a bipartite simple  degeneration. In the following part of this section we denote $\mathbf{H}_{\st}^n=\mathbf{H}_{\st}^n(\mathfrak{X}/C)$.

Denote by $Y_1$ and $Y_2$  the two components of $\pi^{-1}(0)$ as in Definition \ref{def-simpleDeg-bipartite}, and $D=Y_1\cap Y_2$. Let $I=\{a_1,\dots,a_r\}\subset [n+1]=\{1,\dots,n+1\}$, where $r\geq 1$. Following \cite[\S 1.3]{GHH19}, we define  $\mathfrak{X}[n]_I$ to be the subspace of $\mathfrak{X}[n]$ defined by $\{t_{a_1}=\dots=t_{a_r}=0\}$ in the local model. 
Set 
\begin{eqnarray*}
W_{I,0}&=&(Y_1|D)^{[a_1-1]}\times \Spec\ \Bbbk[\{t_j\}_{j\in \{a_1+1,\dots,n+1\}\setminus I}],\\
W_{I,i}&=&(D,N_{D/Y_1})^{[a_{i+1}-a_{i}]}\times \Spec\ \Bbbk[\{t_j\}_{j\in \{1,\dots,a_{i}-1, a_{i+1}+1,\dots,n+1\}\setminus I}],\ 1\leq i\leq r-1,\\
W_{I,r}&=&(D|Y_2)^{[n+1-a_r]}\times \Spec\ \Bbbk[\{t_j\}_{j\in \{1,\dots,a_r-1\}\setminus I}].
\end{eqnarray*}
Recall the action (\ref{eq-action-1}) of $\mathbb{G}_m^n\cong G[n+1]$ on $\mathbb{A}^{n+1}$. There is an induced action of $\mathbb{G}_m^{n-r}$ on 
\[
\mathbb{A}^{n+1-r}=\Spec\ \Bbbk[\{t_j\}_{j\in \{1,\dots,n+1\}\setminus I}]
\]
via regarding $\mathbb{G}_m^{n-r}$ as the subtorus $\{\sigma_{a_1}=\dots=\sigma_{a_r}=0\}$ of $G[n]$, and restricting the action (\ref{eq-action-1}) to it.
For each $i$, $0\leq i\leq r$, there is an obvious induced $\mathbb{G}_m^{n-r}$-action on $W_{I,i}$.
\begin{proposition}\label{prop-expanded-degeneration-central-fiber-components}
The space $\mathfrak{X}[n]_I$ has a decomposition
\begin{equation}\label{eq-expanded-degeneration-central-fiber-components}
	\mathfrak{X}[n]_I=W_{I,0}
	\cup\bigcup_{i=1}^{r-1}W_{I,i}\cup W_{I,r},
\end{equation}
satisfying that the intersection of any two components lies in the singular locus of the fibers of $\pi[n]:\mathfrak{X}[n]\rightarrow \mathbb{A}^1[n]$. Moreover, the decomposition is $G[n]$-equivariant.
\end{proposition}
\begin{proof}
The conclusion follows from a direct comparison of the local models in Proposition \ref{prop-doublepoint-degeneration-local-model} and Remark \ref{rmk-relative-Hilb-space-local-model}.
\end{proof}

\begin{definition}
\begin{enumerate}
	\item[(i)]Denote by $\pi_{D|Y}^{[n]}$ the composition 
$(D|Y)^{[n]}\xrightarrow{\rho_{D|Y}^{[n]}} Y\times \mathbb{A}^n\xrightarrow{\mathrm{pr}_2}\mathbb{A}^n$. Denote by $\mathrm{Hilb}^n((D|Y)^{[n]}/\mathbb{A}^n)$ the Hilbert space of $n$ points on the fibers of $\pi_{D|Y}^{[n]}$. We denote the open subspace of $\mathrm{Hilb}^n((D|Y)^{[n]}/\mathbb{A}^n)$ which parametrizes the objects satisfying the conditions (i) and (ii) in Theorem \ref{thm-stable-locus} by $\mathbf{H}^n_{\st}(D|Y)$. There is an induced action of $\mathbb{G}_m^{n}$ on $\mathbf{H}^n_{\st}(D|Y)$. Similarly we define $\mathbf{H}^n_{\st}(Y|D)$  with a $\mathbb{G}_m^n$-action. 
	\item[(ii)] We define $\mathbf{H}^n_{\st}(D,N):=\mathbf{H}^n_{\st}(D_0|\mathbb{P}_{D,N})$, with an induced $\mathbb{G}_m^n$-action.
	\item[(iii)] For $I\subset [n+1]$, denote by $(\mathbf{H}^n_{\st})_I$  the closed subspace of $\mathbf{H}^n_{\st}$ parametrizing the zero-dimensional subschemes lying in $\mathfrak{X}[n]_I$.
\end{enumerate}
\end{definition}

\begin{corollary}\label{cor-factors-GIT-hilbert-stack}
There is a $G[n]$-equivariant isomorphism
\begin{equation}
	(\mathbf{H}^n_{\st})_I\cong \mathbf{H}^{a_1-1}_{\st}(Y_1|D)\times \prod_{i=1}^{r-1}\mathbf{H}^{a_{i+1}-a_i}_{\st}(D,N_{D/Y_1})\times \mathbf{H}^{n+1-a_r}_{\st}(D|Y_2),
\end{equation}
and an isomorphism
\begin{multline}\label{eq-factors-GIT-hilbert-stack}
	\big[(\mathbf{H}^n_{\st})_I/G[n]\big]\cong \big[\mathbf{H}^{a_1-1}_{\st}(Y_1|D)/G[a_1-1]\big]\\
	\times 
	\prod_{i=1}^{r-1}\big[\mathbf{H}^{a_{i+1}-a_i}_{\st}(D,N_{D/Y_1})/G[a_{i+1}-a_i]\big]\times \big[\mathbf{H}^{n+1-a_r}_{\st}(D|Y_2)/G[n+1-a_r]\big]
\end{multline}
of algebraic stacks over $\Bbbk$.
\end{corollary}
\begin{proof}
By Definition \ref{def-good-degeneration-Hilbert-scheme-points}, the objects parametrized by $\mathbf{H}^n_{\st}$ lies in the smooth loci of the fibers of $\pi[n]$. So conclusion  follows from \ref{prop-expanded-degeneration-central-fiber-components}, and by taking the  quotients by $G[n]$.
\end{proof}


\begin{corollary}
\begin{enumerate}
	\item[(i)] Let $(Y,D)$ be a pair of smooth proper algebraic spaces over $\Bbbk$, such that $D$ is a divisor of $Y$. Then for $n\geq 0$, $\big[\mathbf{H}^{n}_{\st}(Y|D)/G[n]\big]$ is a proper Deligne-Mumford stack over $\Bbbk$. If $Y$ is a projective scheme, $\big[\mathbf{H}^{n}_{\st}(Y|D)/G[n]\big]$ has a projective coarse moduli scheme;
	\item[(ii)] Let $D$ be a smooth proper algebraic space over $\Bbbk$, $N$ a line bundle on $D$. Then for $n\geq 0$, $\big[\mathbf{H}^{n}_{\st}(D,N)/G[n]\big]$ is a proper Deligne-Mumford stack over $\Bbbk$. If $D$ is a projective scheme,  $\big[\mathbf{H}^{n}_{\st}(D,N)/G[n]\big]$ has a projective coarse moduli scheme.
\end{enumerate}
\end{corollary}
\begin{proof}
Apply Corollary \ref{cor-factors-GIT-hilbert-stack} to the standard degeneration $\mathfrak{S}_{Y,D}$ associated with the pair $(Y,D)$. Then
$\big[(\mathbf{H}^n_{\st})_I/G[n]\big]$ is a closed substack of $(\mathcal{I}^n_{\mathfrak{S}_{Y,D}/C})_0$,
the fiber of $(\mathcal{I}^n_{\mathfrak{S}_{Y,D}/C})$ over $0\in C$. By Theorem 
\ref{thm-properness-X-algebraic-space} $(\mathcal{I}^n_{\mathfrak{S}_{Y,D}/C})_0$ is a proper Deligne-Mumford stack over $\Bbbk$. So from the factorization 
(\ref{eq-factors-GIT-hilbert-stack}) we obtain the first statement of (i). The second statement follows from Theorem \ref{thm-properness-X-projective}.

Similarly, applying Corollary \ref{cor-factors-GIT-hilbert-stack}, Theorem \ref{thm-properness-X-algebraic-space} and Theorem \ref{thm-properness-X-projective} to the standard degeneration associated with the pair $(\mathbb{P}(N\oplus \mathcal{O}_D),D)$ we obtain (ii).
\end{proof}

\begin{definition}
We define the \emph{relative Hilbert spaces of $n$ points} associated with $(Y,D)$ (resp. associated with $(D,N)$) to be the Deligne-Mumford stacks
\begin{enumerate}
	\item[(i)] $\mathrm{Hilb}^n(Y,D):=\big[\mathbf{H}^{n}_{\st}(Y|D)/G[n]\big]$;
	\item[(ii)] resp.  $\mathrm{Hilb}^n(D,N):=\big[\mathbf{H}^{n}_{\st}(D,N)/G[n]\big]$.
\end{enumerate}
\end{definition}

\begin{proposition}\label{prop-central-fiber-relativeHilbertscheme}
Let $\pi:\mathfrak{X}\rightarrow C$ be a bipartite simple  degeneration. Let $(\mathcal{I}^n_{\mathfrak{X}/C})_0$ be the central fiber of $\mathcal{I}^n_{\mathfrak{X}/C}$ over $0\in C$. For any subset $I=\{a_1,\dots,a_r\}\subset[n+1]$, let $(\mathcal{I}^n_{\mathfrak{X}/C})_I$ be the closed substack of $(\mathcal{I}^n_{\mathfrak{X}/C})_0$ defined by $t_{a_1}=\dots=t_{a_r}=0$.   Then:
\begin{enumerate}
 	\item[(i)] $(\mathcal{I}^n_{\mathfrak{X}/C})_0$ is a scheme theoretic union of the closed substacks
	\begin{equation}\label{eq-centralfiber-components}
		(\mathcal{I}^{n}_{\mathfrak{X}/C})_{\{i\}},\ 1\leq i\leq n+1;
	\end{equation}
	\item[(ii)] for  $I=\{a_1,\dots,a_r\}\subset[n+1]$, 
	\begin{equation}\label{eq-intersection-strata}
		(\mathcal{I}^n_{\mathfrak{X}/C})_{I}
		\cong (\mathcal{I}^n_{\mathfrak{X}/C})_{\{a_1\}}\times_{(\mathcal{I}^n_{\mathfrak{X}/C})_0}\cdots\times_{(\mathcal{I}^n_{\mathfrak{X}/C})_0}(\mathcal{I}^n_{\mathfrak{X}/C})_{{\{a_r\}}}
	\end{equation}
	In other words, $(\mathcal{I}^n_{\mathfrak{X}/C})_{I}$ is equal to the \emph{scheme theoretic intersection} of the components in (\ref{eq-centralfiber-components}) indexed by $i\in I$.
	\item[(iii)] for  $I=\{a_1,\dots,a_r\}\subset[n+1]$, there is an isomorphism of stacks
	\begin{equation}\label{eq-centralfiber-strata-decomposition}
		(\mathcal{I}^n_{\mathfrak{X}/C})_{I}\cong \mathrm{Hilb}^{a_1-1}(Y_1,D)\times
	\prod_{i=1}^{r-1}\mathrm{Hilb}^{a_{i+1}-a_i}(D,N_{Y_1/D})\times
	 \mathrm{Hilb}^{n+1-a_{r}}(Y_2,D).
	\end{equation}

 \end{enumerate} 

\end{proposition}
\begin{proof}
In the proof of Proposition \ref{prop-flatness-good-degeneration}, we have shown that $q:\mathbf{H}^n_{\st}\rightarrow C[n] $ is flat. Since $(t_0,\dots,t_n)$ is a regular sequence of $C[n]$, $(q^*t_0,\dots,q^* t_n)$ is a regular sequence of  $\mathbf{H}^n_{\st}$ as well. Then the equality of ideals 
\[
(t_{a_1}\cdots t_{a_r})=\bigcap_{i=1}^{r}(t_{a_i})
\]
implies (i) and (ii). Finally (iii) follows from Corollary \ref{cor-factors-GIT-hilbert-stack}.
\end{proof}

\section{Base change property of higher direct images of sheaves on stacks}\label{sec:baseChange}
In this section we show the local constantness of the Euler characteristic of a flat family of sheaves on a tame Deligne-Mumford stack. We follow the argument of \cite[section 8.3]{Ill05}. 
\begin{proposition}\label{prop-base-change}
Let $S$ be scheme. 
Let  
\[
\xymatrix{
	\mathscr{X}' \ar[r]^{h} \ar[d]_{f'} & \mathscr{X} \ar[d]^{f} \\
	\mathscr{Y}' \ar[r]^{g} & \mathscr{Y}
}
\]
be a 2-cartesian diagram of morphisms of algebraic stacks over $S$. Suppose that $f$ is quasi-compact and the diagonal morphism $\Delta_f: \mathscr{X}\rightarrow \mathscr{Y}$ of $f$ is affine. Let $\mathcal{F}$ be a quasi-coherent sheaf on $\mathscr{X}$, and $\mathcal{G}$ a quasi-coherent sheaf on $\mathscr{Y}'$. Suppose that $\mathcal{F}$ and $\mathcal{G}$ are tor-independent over $\mathscr{Y}$. Then there is a natural isomorphism
\[
\mathcal{G}\otimes^L_{\mathscr{Y}} Rf_* \mathcal{F}\xrightarrow{\sim} Rf'_* (\mathcal{G}\otimes^ L_{\mathscr{Y}} \mathcal{F}).
\]
\end{proposition}
If $\mathscr{X}$ is a Deligne-Mumford stack over $S$ and $f$ is separated, the assumption that the diagonal $\Delta_f$ be affine is satisfied.
\begin{proof} 
The required base change property is local on smooth charts of $\mathscr{Y}'$ and $\mathscr{Y}$. Moreover, the diagonal being affine is stable under base change. So without loss of generality we assume that $\mathscr{Y}'$ and $\mathscr{Y}$ are affine schemes. Let $\mathscr{Y}=\Spec(A)$, $\mathscr{Y}'=\Spec(A')$. Since $f$ is quasi-compact, we can find $U\rightarrow \mathscr{X}$  a smooth surjective morphism where $U$ is an affine scheme. Set
\[
U_n:=\underbrace{U\times_{\mathscr{X}}U\times_{\mathscr{X}}\cdots \times_{\mathscr{X}}U}_{n+1}.
\] 
By our assumption $\Delta_f$ is affine, thus from the cartesian diagram
\[
\xymatrix{
	U_{n+1}=U_n\times_{\mathscr{X}}U \ar[d] \ar[r] & U_n\times_{\mathscr{Y}} U \ar[d]\\
	\mathscr{X} \ar[r]^>>>>>>>>>>{\Delta_f} & \mathscr{X}\times_{\mathscr{Y}}\mathscr{X}
}
\]
it follows inductively that $U_n$ is affine. Denote by $\pi_i:U_i\rightarrow \mathscr{X}$ the projection, and by $\check{C}(U_{\bullet},\pi_{\bullet}^*\mathcal{F})$ the Čech complex, i.e. the complex associated with the simplicial abelian group $i\mapsto \Gamma(U_i, \pi^* \mathcal{F})$. Since $f$ is quasi-compact and quasi-separated, $R^if_* \mathcal{F}$ are quasi-coherent for $i\geq 0$. By \cite[Theorem 2.3, Theorem 6.14]{Ols07},
\[
Rf_* \mathcal{F}=f_*\check{C}(U_{\bullet},\pi_{\bullet}^*\mathcal{F}),
\]
where both sides are regarded as  complexes of $A$-modules. Let $U'_i=U\times_{\mathscr{Y}}\mathscr{Y}'$. Then we also have 
\[
Rf'_* \mathcal{F}'=f'_*\check{C}(U'_{\bullet},\pi_{\bullet}^*\mathcal{F}'),
\]
for quasi-coherent sheaves over $\mathscr{X}'$. By the tor-independence of $\mathcal{G}$ and $\mathcal{F}$ over $\mathscr{Y}$, we conclude
\[
\mathcal{G}\otimes^L_{\mathscr{Y}} Rf_* \mathcal{F}=\mathcal{G}\otimes_{\mathscr{Y}} f_*\check{C}(U_{\bullet},\pi_{\bullet}^*\mathcal{F})
=f'_*\check{C}(U_{\bullet},\mathcal{G}\otimes_{\mathscr{Y}}\pi_{\bullet}^*\mathcal{F})
=Rf'_* (\mathcal{G}\otimes^ L_{\mathscr{Y}} \mathcal{F}).
\]

\end{proof}

\begin{remark}
More generally, if  the assumption that $\Delta_f$ is affine is replaced by the weaker one that $f$ is quasi-separated, one can follow the line of the proof of \cite[08IB]{StPr}.
\end{remark}

\begin{theorem}
Let $S$ be a locally noetherian scheme. 
Let $f:\mathscr{X}\rightarrow \mathscr{Y}$ be a proper morphism of  Deligne-Mumford stacks of finite type over $S$. Let $\mathcal{F}$ be a coherent sheaf on $X$.  Suppose that $\mathcal{F}$ is flat over $\mathscr{Y}$.  For a geometric point $\bar{x}\rightarrow \mathscr{X}$ let $G_{\bar{x}}$ (resp. $H_{\bar{x}}$) denote the stabilizer group of $\bar{x}$ (resp. of $f(\bar{x})$), and let $K_{\bar{x}}$ denote the kernel of $G_{\bar{x}}\rightarrow H_{\bar{x}}$. Suppose that for every geometric point $\bar{x}$ of $\mathscr{X}$, the order of the group $K_{\bar{x}}$ is invertible in the field $k(\bar{x})$. 
Then $Rf_* \mathcal{F}$ is perfect.
\end{theorem}
\begin{proof}
The question is local on $\mathscr{Y}$, so we can replace $\mathscr{Y}$ with an affine chart $\Spec(A)$. Then the assumption says that $\mathscr{X}$ is a tame Deligne-Mumford stack proper over $\Spec(A)$ (see the first part of the proof of \cite[Theorem 11.6.5]{Ols16}). By \cite[Theorem 11.6.1]{Ols16}, $R^i f_* \mathcal{F}$ is coherent for any $i\geq 0$. So it remains to show that there exists $n_0$, such that for any $A$-module $M$, $H^i(M\otimes^L Rf_* \mathcal{F})=0$ for $i>n_0$. By Proposition \ref{prop-base-change}, $H^i(M\otimes^L Rf_* \mathcal{F})=R^i f_*(M\otimes_A \mathcal{F})$. Now again by \cite[Theorem 11.6.5]{Ols16}), such a uniform bound $n_0$ exists.
\end{proof}


\begin{corollary}\label{cor-locally-constant}
Let $S$ be a locally noetherian scheme, $\mathscr{X}$ a  \emph{tame} Deligne-Mumford stack proper over $S$. Let $\mathcal{F}$ be a coherent sheaf on $X$, and suppose that $\mathcal{F}$ is flat over $S$. Then the function
\[
s\mapsto \chi(\mathscr{X}_s, \mathcal{F}_s):=\sum_{i=0}^{\infty}(-1)^i H^i(\mathscr{X}_s, \mathcal{F}/\mathfrak{m}_s \mathcal{F})
\]
is locally constant on $S$.
\end{corollary}

\section{Algebraic cobordism of a list of bundles}\label{sec:alg-cobordism}
In this section we generalize the main results of \cite{LeeP12} to algebraic cobordism of a list of vector bundles.
Let $n$ be a natural number, and $\mathbf{r}=(r_1,\dots,r_k)$ a list of natural numbers. 
The case $k=1$ and the case $r_1=\dots=r_k=1$ are treated in \cite{LeeP12}.

Let $\omega_{n,\mathbf{r}}(\Bbbk)$ be the abelian group generated by the triples $(X,\mathbf{E})$, where $X$ is a smooth connected projective scheme of dimension $n$ over $\Bbbk$, $\mathbf{E}=(E_1,\dots,E_k)$ a list of vector bundles, with $E_i$ a vector bundle on $X$ of rank $r_i$, modulo the relations
\begin{equation}\label{eq-doublepoint-relation-list}
[(X,\mathscrbf{E}|_X)]=[(Y_1,\mathscrbf{E}|_{Y_1})]
+[(Y_2,\mathscrbf{E}|_{Y_2})]-[(\mathbb{P}_D,\mathscrbf{E}|_{\mathbb{P}_D})]
\end{equation}
if there exists a double point degeneration $X\overset{\mathfrak{X}}{\rightsquigarrow}Y_1\underset{D}{\cup} Y_2$, with $\mathfrak{X}\rightarrow U$, where $U$ is an open subscheme of $\mathbb{P}^1$, and $\mathscrbf{E}=(\mathscr{E}_1,\dots,\mathscr{E}_k)$ is a list of vector bundles on $\mathfrak{X}$.

\begin{remark}
In the definition of \cite{LevP09}, the base curve of a double point degeneration is $\mathbb{P}^1$, and it was not assumed that all fibers except the central one are smooth. Since the morphism $\mathfrak{X}\rightarrow U$ can always be completed to a projective flat morphism to $\mathbb{P}^1$, the resulted algebraic cobordism (resp. that of vector bundles) are the same as that of \cite{LevP09} (resp. \cite{LeeP12}).
\end{remark}

\begin{definition}
A \emph{partition} $\lambda$ is an unordered tuple $(a_1,\dots,a_l)$ of natural numbers, and we denote $l(\lambda)=l$ and $|\lambda|=\sum_{i=1}^l a_i$. The empty partition is allowed, i.e.  $l=0$. If $\lambda=(a_1,\dots,a_l)$ be a partition, a \emph{subpartition} of $\lambda$ is a partition obtained by deleting some $a_i$'s in $(a_1,\dots,a_l)$.  
\end{definition}

Let $\mathcal{P}_{n,\mathbf{r}}$ be the set of $(k+1)$-tuples of partitions $(\lambda,\mu_1,\dots,\mu_k)$, satisfying that $|\lambda|=n$, $\mu_1\sqcup\dots\sqcup \mu_k$ forms a subpartition of $\lambda$, and $l(\mu_i)\leq r_i$ for $1\leq i\leq k$. For a partition $\lambda=(\lambda_1,\dots,\lambda_{l(\lambda)})$, put $\mathbb{P}^{\lambda}=\mathbb{P}^{\lambda_1}\times\cdots\times \mathbb{P}^{\lambda_{l(\lambda)}}$. For $(\lambda,\mu_1,\dots,\mu_k)\in \mathcal{P}_{n,\mathbf{r}}$, and $m\in \mu_1\sqcup\dots\sqcup \mu_k$, let $L_m$ be the pull-back of the line bundle $\mathcal{O}_{\mathbb{P}^m}(1)$ via the projection $\mathbb{P}^{\lambda}\rightarrow \mathbb{P}^m$. We define a map $
\phi: \mathcal{P}_{n,\mathbf{r}}\rightarrow \omega_{n,\mathbf{r}}$
by
\[
\phi(\lambda,\mu_1,\dots,\mu_k)=[\mathbb{P}^{\lambda},E_1,\dots,E_k].
\]
where
\[
E_i=\mathcal{O}_{\mathbb{P}^{\lambda}}^{r-l(\mu_i)}\oplus\bigoplus_{m\in \mu_i}L_m,\ 1\leq i\leq k.
\]
Let $\mathcal{Q}_{n,\mathbf{r}}$ be the set of $(k+1)$-tuples of partitions $(\nu,\mu_1,\dots,\mu_k)$, satisfying $|\nu|+\sum_{i=1}^k |\mu_i|=n$, and that the largest part of 
$\mu_i \leq r_i$, for $1\leq i\leq k$. There is a bijective map
\[
\epsilon: \mathcal{Q}_{n,\mathbf{r}}\rightarrow \mathcal{P}_{n,\mathbf{r}}
\]
\[
(\nu,\mu_1,\dots,\mu_k)\mapsto (\nu\cup \mu_1^t\cup\dots\cup \mu_k^t,\mu_1^t,\dots,\mu_k^t).
\]

Consider the ring $\mathbb{Q}[u_1,\dots,u_n,v_1^{(1)},\dots,v_{r_1}^{(1)},\dots,v_1^{(k)},\dots,v_{r_k}^{(k)}]$ of $n+|\mathbf{r}|$ variables.
Let $\mathcal{C}_{n,\mathbf{r}}$ be the vector space of degree $n$  polynomials  of this ring.
To each
\begin{equation}\label{eq-expression-elementsof-Q}
 (\lambda,\mu_1,\dots,\mu_k)=(1^{l_1}\cdots n^{l_n},1^{m_{1,1}}\cdots r_1^{m_{1,r_1}},\dots,
 1^{m_{k,1}}\cdots r_k^{m_{k,r_k}})
 \end{equation}
 in $\mathcal{Q}_{n,\mathbf{r}}$, we attach a monomial in $\mathcal{C}_{n,\mathbf{r}}$ 
 \begin{equation}\label{eq-basis-chernpolynomial}
 \mathsf{C}(\lambda,\mu_1,\dots,\mu_k)=
 u_1^{l_1}\cdots u_n^{l_n}\cdot(v_{1}^{(1)})^{m_{1,1}}\cdots (v_{r_1}^{(1)})^{m_{1,r_1}}
 \cdot\cdots\cdot(v_{1}^{(k)})^{m_{k,1}}\cdots (v_{r_k}^{(k)})^{m_{k,r_k}}.
 \end{equation}
 Since the monomials of this form form a basis of $\mathcal{C}_{n,\mathbf{r}}$, the set  $\mathcal{Q}_{n,\mathbf{r}}$ corresponds to a basis of $\mathcal{C}_{n,\mathbf{r}}$ in a natural way. 

 Given a representative $(X,\mathbf{E})$ of some class in $\omega_{n,\mathbf{r}}(\Bbbk)$, we evaluate the monomial (\ref{eq-basis-chernpolynomial}) by
 \begin{equation}\label{eq-evaluate-chernpolynomial}
 \int_{X}c_1(T_X)^{l_1}\cdots c_{n}(T_X)^{l_n}\cdot c_1(E_1)^{m_{1,1}}\cdots 
 c_{r_1}(E_1)^{m_{1,r_1}}\cdots c_1(E_k)^{m_{k,1}}\cdots c_{r_k}(E_k)^{m_{k,r_k}}.
 \end{equation}
This evaluation extends to $\mathcal{C}_{n,\mathbf{r}}$ linearly. For $f\in \mathcal{C}_{n,\mathbf{r}}$, we denote the corresponding integrand by $\Phi_f(\mathbf{E})$. 
 \begin{proposition}
The evaluation (\ref{eq-evaluate-chernpolynomial}) factors through the relations (\ref{eq-doublepoint-relation-list}), and thus define a bilinear pairing 
\[
\rho:\omega_{n,\mathbf{r}}\otimes \mathcal{C}_{n,\mathbf{r}}\rightarrow \mathbb{Z},\
([X,\mathbf{E}],f)\mapsto \int_X \Phi_{f}(\mathbf{E}).
\]
 \end{proposition}
\begin{proof}
One can mimic the proof of \cite[Proposition 5]{LeeP12}, namely the $k=1$ case. Here we take a different way. In Appendix \ref{sec:identity-chern-doublepoint}, for a double point degeneration $X\overset{\mathfrak{X}}{\rightsquigarrow}Y_1\underset{D}{\cup} Y_2$  we show an identity (\ref{eq-chern-identity-doublepointdegeneration}) of Chern classes of the tangent bundles of $X$, $Y_1$, $Y_2$ and $D$ in $\mathrm{CH}(\mathfrak{X})$, which is stronger than the numerical equivalence that we need. So the conclusion follows.
\end{proof}

\begin{theorem}\label{thm-algebraic-cobordism-listofbundles}
The pairing $\rho: \omega_{n,\mathbf{r}}\otimes \mathcal{C}_{n,\mathbf{r}}\rightarrow \mathbb{Q}$ is perfect, and $\phi(\mathcal{P}_{n,\mathbf{r}})$ form a $\mathbb{Q}$-basis of $\omega_{n,\mathbf{r}}$.
\end{theorem}
\begin{proof} When $k=1$ this is \cite[Theorem 1 and Theorem 4]{LeeP12}. Their proof (see pages 1089-1090 and 1096-1097 of \cite{LeeP12}) generalizes to arbitrary $k$  directly. We give a sketch. For $1\leq i\leq k$ and an element expressed as (\ref{eq-expression-elementsof-Q}) we define its $v^{(i)}$-degree as
\[
\deg_{v^{(i)}}(1^{l_1}\cdots n^{l_n},1^{m_{1,1}}\cdots r_1^{m_{1,r_1}},\dots,
 1^{m_{k,1}}\cdots r_k^{m_{k,r_k}})=(m_{i,1},\dots,m_{i,r_i})\in \mathbb{Z}_{\geq 0}^{r_i}.
\]
For any $r\geq 1$ we equip $\mathbb{Z}_{\geq 0}^{r}$ with the lexicographic order: $(m_1,\dots,m_r)<(m'_1,\dots,m'_r)$ if there exists $j$ such that $m_a=m'_a$ for $a>j$ and $m_j<m'_j$. Then the argument of \cite[proof of Lemma 7]{LeeP12} shows that if there exists $i$ such that 
\[
\deg_{v^{(i)}}(\lambda,\mu_1,\dots,\mu_k)<\deg_{v^{(i)}}(\lambda,\mu'_1,\dots,\mu'_k),
\]
then
\[
\rho\big(\phi\circ \epsilon(\lambda,\mu_1,\dots,\mu_k),(\lambda,\mu'_1,\dots,\mu'_k)
\big)=0.
\]
So the matrix 
\begin{equation}
	\mathbf{M}_{(\lambda,\mu_1,\dots,\mu_k),(\lambda,\mu'_1,\dots,\mu'_k)}:=\rho\big(\phi\circ \epsilon(\lambda,\mu_1,\dots,\mu_k),(\lambda,\mu'_1,\dots,\mu'_k)
\big)
\end{equation}
is a block triangular matrix. Then the argument of \cite[proof of Proposition 8]{LeeP12} shows that $\mathbf{M}$ is nonsingular. So the map
\[
 \omega_{n,\mathbf{r}}\otimes \mathbb{Q}\rightarrow \mathcal{C}_{n,\mathbf{r}}
\]
induced by $\rho$ is surjective. It remains to show
\[
\mathrm{rank}\ \omega_{n,\mathbf{r}}\otimes \mathbb{Q}\leq \mathrm{rank}(\mathcal{C}_{n,\mathbf{r}})=\sharp(\mathcal{P}_{n,r_1,\dots,r_k}).
\]
For a representative $[Y,E_1,\dots,E_k]$, by a repeated application of \cite[Lemma 13]{LeeP12}, there exists a birational morphism $p:\hat{Y}\rightarrow Y$ such that each $p^* E_i$ has a filtration by subbundles such that the subquotients are line bundles. Then the proof of \cite[Proposition 12]{LeeP12} and  the argument of \cite[\S 3.2]{LeeP12} reduces the problem to the case $r_1=\cdots=r_k=1$, which is treated in \cite[\S 2]{LeeP12}. 
\end{proof}

\section{Tautological sheaves and the degeneration formula}\label{sec:tautological-sheaves-degeneration-formula}
\begin{proposition}\label{prop-smoothlocus-quotient}
Let $C$ be a smooth curve over $\Bbbk$, $\mathfrak{X}$ a smooth algebraic space over $\Bbbk$, and $\pi:\mathfrak{X}\rightarrow C$ a bipartite simple  degeneration. Denote by $\mathfrak{X}[n]^{sm}$ the locus of points that are smooth on the fibers of $\pi[n]:\mathfrak{X}[n]\rightarrow C[n]$.
Then the $G[n]$-action on $\mathfrak{X}[n]^{sm}$ is affine, and the geometric quotient $\mathfrak{X}[n]^{sm}/G[n]$ is naturally isomorphic to $\mathfrak{X}$. In particular, $\mathcal{I}_{\mathfrak{X}/C}^1=\mathfrak{X}$.
\end{proposition}
\begin{proof}
 Let $U\rightarrow \mathfrak{X}$ be an étale chart such that $U$ is an affine scheme and there is an étale morphism $f: U\rightarrow	V=\Spec\ \Bbbk[x,y,z,\dots,t]/(xy-t)$. Then there  is a $G[n]$-equivariant étale morphism $f[n]: U[n]\rightarrow  V[n]$. So we need only compute $V[n]^{sm}/G[n]$. Recall that $V[n]$ is the closed subscheme of 
\[
\Spec\ \Bbbk[x,y,z,\dots,t_1,\dots,t_{n+1}]\times \prod_{i=1}^n\mathrm{Proj}\ \Bbbk[u_i,v_i],
\]
defined by 
\begin{eqnarray*}
	\begin{cases}
	u_1 x=v_1 t_1,\\
	u_i v_{i-1}=v_i u_{i-1} t_i,\ 1<i\leq n,\\
	v_n y=u_n t_{n+1}.
	\end{cases}
	\end{eqnarray*}
By Lemma \ref{lemma-fiber-smoothness} $V[n]^{sm}$ is the open subscheme 
\[
\{x\neq 0\}\cup\bigcup_{k=1}^{n}\{u_k\neq 0,v_k\neq 0\}\cup \{y\neq 0\}.
\]
Let $F_0=\{x\neq 0\}$, $F_k=\{u_k\neq 0,v_k\neq 0\}$, $1\leq k\leq n$, and $F_{n+1}=\{y\neq 0\}$. Then the $F_j$'s are affine and $G[n]$-invariant for $0\leq j\leq n+1$. So the $G[n]$-action on $V[n]^{sm}$ is affine, and thus the $G[n]$-action on $U[n]^{sm}$ is also affine, for $U[n]^{sm}=f[n]^{-1}(V[n])$ and $f[n]$ are affine. 

We have $F_0=\Spec(A_0)$, where
\begin{align*}
A_0=\frac{\Bbbk[x,\frac{1}{x},y,z,\dots,t_1,\dots,t_{n+1},u_1,\dots,u_n]}{(u_1-\frac{t_1}{x_1},\dots,u_i-u_{i-1}t_i,\dots,u_{n}-u_{n-1}t_n,y-u_n t_{n+1})},
\end{align*}
with the $G[n]$-action
\begin{eqnarray*}
&&\sigma.u_i=\sigma_1\cdots \sigma_i u_i,\ 1\leq i\leq n,\\
&& \sigma. t_i=\sigma_i t_i,\ 1\leq i\leq n+1;
\end{eqnarray*}
and $F_k=\Spec(A_k)$, where
\begin{eqnarray*}
A_k=\frac{\Bbbk[x,y,z,\dots,t_1,\dots,t_{n+1},v_1,v_2,\dots,v_{k-1},v_k,\frac{1}{v_k},u_{k+1},u_{k+2},\dots,u_n]}{\big(x-t_1\cdots t_k v_k,\underbrace{v_i-t_{i+1}\cdots t_{k} v_k}_{1\leq i\leq k-1},
\underbrace{u_i-\frac{t_{k+1}\cdots t_{i}}{v_k}}_{k+1\leq i\leq n},y-\frac{t_{k+1}\cdots t_{n+1}}{v_k}\big)}
\end{eqnarray*}
with the $G[n]$-action
\begin{eqnarray*}
&&\sigma.v_i=\frac{v_i}{\sigma_1\cdots \sigma_i},\ 1\leq i\leq k,\\
&&\sigma.u_i=\sigma_1\cdots \sigma_i u_i,\ k+1\leq i\leq n,\\
&&\sigma.t_i=\sigma_i t_i,\ 1\leq i\leq n+1;
\end{eqnarray*}
and $F_{n+1}=\Spec(A_{n+1})$, where
\begin{eqnarray*}
A_{n+1}=\frac{\Bbbk[x,y,\frac{1}{y},z,\dots,t_1,\dots,t_{n+1},v_1,\dots,v_n]}{(x-v_1 t_1,\dots,v_{i-1}-v_{i}t_i,\dots,v_{n-1}-v_{n}t_n,v_n-\frac{t_{n+1}}{y})}
\end{eqnarray*}
with the $G[n]$-action
\begin{eqnarray*}
&&\sigma.v_i=\frac{v_i}{\sigma_1\cdots \sigma_i},\ 1\leq i\leq n,\\
&&\sigma.t_i=\sigma_i t_i,\ 1\leq i\leq n+1.
\end{eqnarray*}

We compute the invariant subrings $A_k^{G[n]}$:
\begin{eqnarray*}
A_0^{G[n]}&=&\big(\Bbbk[x,\frac{1}{x},t_1,\dots,t_k]/(y-\frac{t_1\cdots t_n}{x})\big)^{G[n]}
=\Bbbk[x,\frac{1}{x},y,t_1\cdots t_n]/(xy-t_1\cdots t_k),\\
A_k^{G[n]}&=&\left(\frac{\Bbbk[x,y,z,\dots,t_1,\dots,t_{n+1},v_k,\frac{1}{v_k}]}{\big(x-t_1\cdots t_k v_k,y-\frac{t_{k+1}\cdots t_{n+1}}{v_k}\big)}\right)^{G[n]}=\Bbbk[x,y,z,\dots],\\
A_{n+1}^{G[n]}&=&\big(\Bbbk[x,y,\frac{1}{y},z,\dots,t_1,\dots,t_{n+1}]/(x-\frac{t_1\cdots t_{n+1}}{y})\big)^{G[n]}\\
&=&\Bbbk[x,y,\frac{1}{y},t_1\cdots t_n]/(xy-t_1\cdots t_k).
\end{eqnarray*}
So $V[n]^{sm}/G[n]\cong V$. All the involved isomorphisms are natural, thus varying $U$ and using étale descent we obtain a natural isomorphism $\mathfrak{X}[n]^{sm}/G[n]=\mathfrak{X}$.

By definition $\mathcal{I}_{\mathfrak{X}/C}^1=[\mathbf{H}_{\st}^1/G[1]]$, and $\mathbf{H}_{\st}^1$ is the open subspace of $\mathfrak{X}[1]$ consisting of the points which are smooth along the fibers of $\mathfrak{X}[1]\rightarrow C[1]$. In the one point case the $G[1]$-action is free. So the resulted stack is the geometric quotient.
\end{proof}

By the construction of $\mathcal{I}_{\mathfrak{X}/C}^n$ and Proposition \ref{prop-smoothlocus-quotient}, the  diagram
\begin{equation}
	\xymatrix{
	\mathcal{Z}_n \ar[r] \ar[d] & \mathfrak{X}[n]^{sm} \\
	\mathbf{H}_{\st}^n & 
	}
\end{equation}
induces a diagram
\begin{equation}
	\xymatrix{
	\big[\mathcal{Z}_n/G[n]\big] \ar[r]^>>>>>{q} \ar[d]_{p} & \big[\mathfrak{X}[n]^{sm}/G[n]\big] \ar[r]^>>>>>{r} & \mathfrak{X} \\
	\mathcal{I}_{\mathfrak{X}/C}^n  & ,
	}
\end{equation}
where $r$ is the projective from $\big[\mathfrak{X}[n]^{sm}/G[n]\big]$ to its coarse moduli $\mathfrak{X}$, as shown in Proposition \ref{prop-smoothlocus-quotient}. 
Let $\mathbf{q}=r\circ q$. For a vector bundle $E$ on $\mathfrak{X}$, define
\[
E^{[n]}=p_* \mathbf{q}^* E.
\]
Since $p$ is  representable and finite, $E^{[n]}$ is a vector bundle on  $\mathcal{I}_{\mathfrak{X}/C}^n$.\\

Let $Y$ be a smooth proper algebraic space over $\Bbbk$, and $D$ a closed smooth subspace of $Y$. Let $F$ be a vector bundle on $Y$. Denote by $\tilde{F}$ the pullback of $E$ onto $\mathfrak{S}_{Y,D}$ via
\[
 \mathfrak{S}_{Y,D}=\mathrm{Bl}_{D}(Y\times \mathbb{P}^1)\rightarrow Y\times \mathbb{P}^1\xrightarrow{p_1} Y.
\]
Apply the above construction to $\mathfrak{X}=\mathfrak{S}_{Y,D}\rightarrow \mathbb{P}^1$. Then we define
\[
 F^{[n]}_{Y,D}:=\tilde{F}^{[n]}|_{\mathrm{Hilb}^n(Y,D)},
 \] 
and
\[
F^{[n]}_{D,N}=\tilde{F}^{[n]}|_{\mathrm{Hilb}^n(D,N)}.
\]

\begin{theorem}[Degeneration formula]\label{thm-degeneration-formula}
Let $\Bbbk$ be a field of characteristic $0$.
Let $\mathfrak{X}\rightarrow C$ be a bipartite simple  degeneration,  $\pi^{-1}(0)=Y_1\cup Y_2$, and $D=Y_1\cap Y_2$. Let $E$ and $F$ be two vector bundles on $\mathfrak{X}$. Let $N$ be the normal bundle of $D$ in $Y_1$. Let $\xi$ by a $\Bbbk$-point of $C$ away from $0\in C$, and $\mathfrak{X}_{\xi}$ the fiber over $\xi$. Let $E_{\xi}$ be the restriction of $E$ to $\mathfrak{X}_{\xi}$. Let $E_i$ be the restriction of $E$ to $Y_i$ for $i=1,2$. Let $E_0$ be the restriction of $E$ to $D$. Similarly for $F$.
 Then
\begin{eqnarray}\label{eq-degeneration-formula}
&&\log\Big(1+\sum_{n=1}^{\infty}\chi\big(\Lambda_{-u}(E_{\xi}^{[n]}),
\Lambda_{-v}(F_{\xi}^{[n]})\big)Q^n\Big)\notag\\
&=&\log\Big(1+\sum_{n=1}^{\infty}\chi\big(\Lambda_{-u}(E_{1,Y_1,D}^{[n]}),
\Lambda_{-v}(F_{1,Y_1,D}^{[n]})\big)Q^n\Big)\notag\\
&&+\log\Big(1+\sum_{n=1}^{\infty}\chi\big(\Lambda_{-u}(E_{2,Y_2,D}^{[n]}),
\Lambda_{-v}(F_{2,Y_2,D}^{[n]})\big)Q^n\Big)\notag\\
&&-\log\Big(1+\sum_{n=1}^{\infty}\chi\big(\Lambda_{-u}(E_{0,D,N}^{[n]}),
\Lambda_{-v}(F_{0,D,N}^{[n]})\big)Q^n\Big).
\end{eqnarray}
\end{theorem}
\begin{proof}
Note that $\Lambda_{-u}(\cdot)$ is multiplicative, and for a vector bundle taking dual commutes with pulling back. Let $\varphi$ be the composition $\varphi:\mathcal{I}^n_{\mathfrak{X}/C}\rightarrow C[n]\rightarrow C$.
Apply Corollary \ref{cor-locally-constant} to $\varphi:\mathcal{I}^n_{\mathfrak{X}/C}\rightarrow C$. The fiber  $\varphi^{-1}(\xi)$ is isomorphic to $\mathrm{Hilb}^n(\mathfrak{X}_{\xi})$. 
We obtain
\begin{equation}\label{eq-proof-thm-degeneration-formula-1}
\chi\big(\mathrm{Hilb}^n(\mathfrak{X}_{\xi}),E^{[n]}|_{\varphi^{-1}(\xi)}\big)
=\chi\big((\mathcal{I}_{\mathfrak{X}/C}^{n})_0,E^{[n]}|_{\varphi^{-1}(0)}\big).
\end{equation}
We decompose the fiber $\mathcal{I}_{\mathfrak{X}/C}^n$ as (\ref{eq-centralfiber-components}). The description (\ref{eq-intersection-strata}) of the scheme theoretic intersections of the the components in (\ref{eq-centralfiber-components}) enable us to express the right handside of (\ref{eq-proof-thm-degeneration-formula-1}) as
\begin{equation}\label{eq-proof-thm-degeneration-formula-2}
	\chi\big((\mathcal{I}_{\mathfrak{X}/C}^{n})_0,E^{[n]}|_{\varphi^{-1}(0)}\big)
=\sum_{\emptyset\neq I\subset [n+1]} (-1)^{|I|-1}
\mathcal{\mathfrak{X}}\big((\mathcal{I}_{\mathfrak{X}/C}^{n})_I,E^{[n]}|_{(\mathcal{I}_{\mathfrak{X}/C}^{n})_I}\big).	
\end{equation}
By (\ref{eq-centralfiber-strata-decomposition}) and the definition of $E^{[n]}$, we have
\begin{equation}\label{eq-proof-thm-degeneration-formula-3}
	E^{[n]}|_{(\mathcal{I}_{\mathfrak{X}/C}^{n})_I}\cong
	\mathrm{pr}_0^* E_1^{[a_{1}-1]}\oplus
	\bigoplus_{i=1}^{r-1}\mathrm{pr}_i^* E_0^{[a_{i+1}-a_i]}\oplus 
	\mathrm{pr}_{r}^* E_2^{[n+1-a_r]}.
\end{equation}
Replacing $E^{[n]}$ in (\ref{eq-proof-thm-degeneration-formula-1}) and (\ref{eq-proof-thm-degeneration-formula-1}) by 
$\lambda_{-u}E^{[n]\vee}\otimes \lambda_{-v}F^{[n]}$, then use (\ref{eq-proof-thm-degeneration-formula-3}) and that $\Lambda_{-u}(\cdot)$ is multiplicative, we obtain
\begin{eqnarray}\label{eq-proof-thm-degeneration-formula-4}
&&\chi\big(\Lambda_{-u}(E_{\xi}^{[n]}),\Lambda_{-v}(F_{\xi}^{[n]})\big)\notag\\
&=&\sum_{\emptyset\neq I\subset [n+1]} (-1)^{|I|-1}
\chi\Big((\mathcal{I}_{\mathfrak{X}/C}^{n})_I,
\mathrm{pr}_0^*(\lambda_{-u}E_1^{[a_1-1]\vee}\otimes \lambda_{-v}F_1^{[a_1-1]})\otimes\notag\\
&&	\bigotimes_{i=1}^{r-1}
	\mathrm{pr}_i^*(\lambda_{-u}E_0^{[a_{i+1}-a_i]\vee}\otimes \lambda_{-v}F_0^{[a_{i+1}-a_i]})\otimes 
	\mathrm{pr}_{r}^*(\lambda_{-u}E_2^{[n+1-a_r]\vee}\otimes \lambda_{-v}F_2^{[n+1-a_r]}) \Big).
\end{eqnarray}
Apply Proposition \ref{prop-base-change}  successively to the projection from the first $i$ factors of (\ref{eq-centralfiber-strata-decomposition}) to the first $i-1$ factors, $i=r+1,r,\dots,2$, we obtain
\begin{eqnarray}\label{eq-proof-thm-degeneration-formula-5}
&&\chi\Big((\mathcal{I}_{\mathfrak{X}/C}^{n})_I,
\mathrm{pr}_0^*(\lambda_{-u}E_1^{[a_1-1]\vee}\otimes \lambda_{-v}F_1^{[a_1-1]})\otimes\notag\\
&&	\bigotimes_{i=1}^{r-1}
	\mathrm{pr}_i^*(\lambda_{-u}E_0^{[a_{i+1}-a_i]\vee}\otimes \lambda_{-v}F_0^{[a_{i+1}-a_i]})\otimes 
	\mathrm{pr}_{r}^*(\lambda_{-u}E_2^{[n+1-a_r]\vee}\otimes \lambda_{-v}F_2^{[n+1-a_r]}) \Big)\notag\\
&=& \chi\big(\mathrm{Hilb}^{a_1-1}(Y_1,D),\lambda_{-u}E_{1,Y_1,D}^{[a_1-1]\vee}\otimes \lambda_{-v}F_{1,Y_1,D}^{[a_1-1]}\big)\notag\\
&&\cdot\prod_{i=1}^{r-1}\chi\big(\mathrm{Hilb}^{a_{i+1}-a_i}(D,N_{Y_1/D}),\lambda_{-u}E_{0,D,N}^{[a_{i+1}-a_i]\vee}\otimes \lambda_{-v}F_{0,D,N}^{[a_{i+1}-a_i]})\notag\\
&&\cdot\chi\big(\mathrm{Hilb}^{n+1-a_{r}}(Y_2,D),\lambda_{-u}E_{2,Y_2,D}^{[n+1-a_r]\vee}\otimes \lambda_{-v}F_{2,Y_2,D}^{[n+1-a_r]}).
\end{eqnarray}
(\ref{eq-proof-thm-degeneration-formula-4}) and (\ref{eq-proof-thm-degeneration-formula-5})
together yield
\begin{eqnarray}\label{eq-proof-thm-degeneration-formula-6}
&&\chi\big(\Lambda_{-u}(E_{\xi}^{[n]}),\Lambda_{-v}(F_{\xi}^{[n]})\big)\notag\\
&=&\sum_{\emptyset\neq I\subset [n+1]} (-1)^{|I|-1}	\chi\big(\lambda_{-u}E_{1,Y_1,D}^{[a_1-1]},\lambda_{-v}F_{1,Y_1,D}^{[a_1-1]}\big)
\cdot\prod_{i=1}^{r-1}\chi\big(\lambda_{-u}E_{0,D,N}^{[a_{i+1}-a_i]},\lambda_{-v}F_{0,D,N}^{[a_{i+1}-a_i]})\notag\\
&&\cdot\chi\big(\lambda_{-u}E_{2,Y_2,D}^{[n+1-a_r]},\lambda_{-v}F_{2,Y_2,D}^{[n+1-a_r]}).
\end{eqnarray} 
By Lemma \ref{lem-intro-inc-exc}, the collection of (\ref{eq-proof-thm-degeneration-formula-6}) for $n$ running through the natural numbers, is equivalent to (\ref{eq-degeneration-formula}). The proof is completed.
\end{proof}

\begin{theorem}[Inclusion-exclusion principle]\label{thm-inc-exc-principle}
Assume the assumptions and notations as Theorem \ref{thm-degeneration-formula}.
Let $\mathbb{P}_D=\mathbb{P}(N_{Y_1/D}\oplus \mathcal{O}_D)$, and $E|_{\mathbb{P}_D}$ the pullback of $E$ via $\mathbb{P}_D\rightarrow D\hookrightarrow \mathfrak{X}$.
Then
\begin{eqnarray}\label{eq-inc-exc-principle}
&&\log\Big(1+\sum_{n=1}^{\infty}\chi\big(\Lambda_{-u}(E|_{\mathfrak{X}_{\xi}}^{[n]}),
\Lambda_{-v}(F|_{\mathfrak{X}_{\xi}}^{[n]})\big)Q^n\Big)
+\log\Big(1+\sum_{n=1}^{\infty}\chi\big(\Lambda_{-u}(E|_{\mathbb{P}_D}^{[n]}),
\Lambda_{-v}(F|_{\mathbb{P}_D}^{[n]})\big)Q^n\Big)\notag\\
&=&\log\Big(1+\sum_{n=1}^{\infty}\chi\big(\Lambda_{-u}(E|_{Y_1}^{[n]}),
\Lambda_{-v}(F|_{Y_1}^{[n]})\big)Q^n\Big)
+\log\Big(1+\sum_{n=1}^{\infty}\chi\big(\Lambda_{-u}(E|_{Y_2}^{[n]}),
\Lambda_{-v}(F|_{Y_2}^{[n]})\big)Q^n\Big).\notag\\
\end{eqnarray}
\end{theorem}
\begin{proof}
Apply (\ref{eq-degeneration-formula}) to $\mathfrak{X}\rightarrow \mathbb{P}^1$, and the standard degenerations associated with $(Y_1,D)$, $(Y_2,D)$ and $(\mathbb{P}_D,D_0)$ respectively. Sum the 1st and the 4th resulted equalities, and subtract the 2nd and the 3rd. The terms involving the relative Hilbert spaces cancel, and  we obtain (\ref{eq-inc-exc-principle}).
\end{proof}
As a consequence we obtain the existence of \emph{universal polynomials}:
\begin{corollary}\label{cor-universalSeries}
Let $\Bbbk$ be a field of characteristic 0.
Suppose given natural numbers $d$, $r_1$ and $r_2$. Then 
\begin{enumerate}
	\item[(i)] there exists a series of polynomials $f_{i,j,k}\in \mathcal{C}_{d,r,s}$, for $i\geq 1$ and $j,k\geq 0$, such that for any smooth proper algebraic space $X$ of pure dimension $d$ and vectors bundles $E$ and $F$, with $\mathrm{rank}(E)=r$ and $\mathrm{rank}(F)=s$, we have
	\begin{equation}\label{eq-universalSeries-1}
		1+\sum_{n=1}^{\infty}\chi(\Lambda_{-u}E^{[n]},\Lambda_{-v}F^{[n]})Q^n
		=\exp\Big(\sum_{i=1}^{\infty}\sum_{j=0}^{\infty}\sum_{k=0}^{\infty} Q^i u^j v^k 
		\int_X \Phi_{f_{i,j,k}}(E,F)
		\Big);
	\end{equation}
	\item[(ii)] given a series of polynomials $f_{i,j,k}\in \mathcal{C}_{d,r,s}$, for $i\geq 1$ and $j,k\geq 0$, to verify that (\ref{eq-universalSeries-1}) holds for all smooth proper algebraic space $X$ of pure dimension $d$ over $\Bbbk$ and vector bundles $E$ and $F$, with $\mathrm{rank}(E)=r$ and $\mathrm{rank}(F)=s$, it suffices to verify it for all triples in $\mathcal{P}_{d,r,s}$.
\end{enumerate}
\end{corollary}
\begin{proof}
By Theorem \ref{thm-inc-exc-principle}, the series
\[
 \log\big(1+\sum_{n=1}^{\infty}\chi(\Lambda_{-u}E^{[n]},\Lambda_{-v}F^{[n]})Q^n\big),
 \] 
 and thus each coefficient of $Q^i u^j v^k$, 
 factors through the double point relation (\ref{eq-doublepoint-relation-list}), so by the first statement of Theorem \ref{thm-algebraic-cobordism-listofbundles} there exists $f_{i,j,k}\in \mathcal{C}_{d,r,s}$ such that (\ref{eq-universalSeries-1}) holds for all smooth projective schemes $X$ of pure dimension $d$ over $\Bbbk$ and vectors bundles $E$ and $F$ with ranks $r$ and $s$ respectively. 

Now let $X$ be a smooth proper algebraic space of pure dimension $d$ over $\Bbbk$. 
Suppose that $Z$ is a smooth closed  subspace of $X$. 
By the proof of \cite[Lemma 5.1]{LevP09}, 
there exists a double point  degeneration $X\overset{\mathfrak{X}}{\rightsquigarrow}Y_1\cup_D Y_2$ over $\mathbb{P}^1$, where $Y_1$ is isomorphic to the blowup $X_Z$ of $X$ along $Z$, and $Y_2$ and $\mathbb{P}_D$ are projective bundles over smooth proper algebraic spaces of lower dimensions.
Thus Theorem \ref{thm-inc-exc-principle} yields
\begin{eqnarray}\label{eq-cor-universalSeries-1}
&&\log\Big(1+\sum_{n=1}^{\infty}\chi\big(\Lambda_{-u}(E|_{X}^{[n]}),
\Lambda_{-v}(F|_{X}^{[n]})\big)Q^n\Big)\notag\\
&=&\log\Big(1+\sum_{n=1}^{\infty}\chi\big(\Lambda_{-u}(E|_{X_Z}^{[n]}),
\Lambda_{-v}(F|_{X_Z}^{[n]})\big)Q^n\Big)
+\log\Big(1+\sum_{n=1}^{\infty}\chi\big(\Lambda_{-u}(E|_{Y_2}^{[n]}),
\Lambda_{-v}(F|_{Y_2}^{[n]})\big)Q^n\Big)\notag\\
&&-\log\Big(1+\sum_{n=1}^{\infty}\chi\big(\Lambda_{-u}(E|_{\mathbb{P}_D}^{[n]},
\Lambda_{-u}(F|_{\mathbb{P}_D}^{[n]})\Big).
\end{eqnarray}
Other other hand, by Theorem \ref{thm-chern-identity-doublepointdegeneration} we have
\begin{eqnarray}
&& \sum_{i=1}^{\infty}\sum_{j=0}^{\infty}\sum_{k=0}^{\infty} Q^i u^j v^k 
		\int_X \Phi_{f_{i,j,k}}(E|_X,F|_X)\notag\\
&=& \sum_{i=1}^{\infty}\sum_{j=0}^{\infty}\sum_{k=0}^{\infty} Q^i u^j v^k 
		\int_{X_Z} \Phi_{f_{i,j,k}}(E|_{X_Z},F|_{X_Z})+
	\sum_{i=1}^{\infty}\sum_{j=0}^{\infty}\sum_{k=0}^{\infty} Q^i u^j v^k 
		\int_{Y_2} \Phi_{f_{i,j,k}}(E|_{Y_2},F|_{Y_2})\notag\\
	&&-	\sum_{i=1}^{\infty}\sum_{j=0}^{\infty}\sum_{k=0}^{\infty} Q^i u^j v^k 
		\int_{\mathbb{P}_D} \Phi_{f_{i,j,k}}(E|_{\mathbb{P}_D},F|_{\mathbb{P}_D}).
\end{eqnarray}
 Hence (\ref{eq-universalSeries-1}) holds for $(X,E,F)$ if it holds for the triples $(X_Z,E|_{X_Z},F|_{X_Z})$, $(Y_2,E|_{Y_2},F|_{Y_2})$, and $(\mathbb{P}_D,E|_{\mathbb{P}_D},F|_{\mathbb{P}_D})$. 
 Using Theorem \ref{thm-Moishezon-algsp}, and a dévissage, (\ref{eq-universalSeries-1})  is reduced to the projective case. So (i) is proved.

(ii) follows from (i) and the second statement of Theorem \ref{thm-algebraic-cobordism-listofbundles}.
\end{proof}

Corollary \ref{cor-universalSeries} (ii) and the definition of $\mathcal{P}_{d,1,1}$ yields the following:
\begin{corollary}\label{cor-reduction-to-toric}
Let $\Bbbk$ be a field of characteristic $0$.
Assume that  Conjecture \ref{conj-1} holds for all triples of the form
\begin{equation}\label{eq-cobordism-representative}
	(\mathbb{P}^{d_1}\times\cdots \times\mathbb{P}^{d_l},p_{i}^* \mathcal{O}(k_1),p_{j}^*\mathcal{O}(k_2))
\end{equation}
where $d_1+\dots+d_l=d$, $k_1,k_2=0$ or $1$, and $1\leq i,j\leq l$, and $p_i$ is the projection to the $i$-th factor. Then it holds for all smooth proper algebraic spaces $X$ of dimension $d$ over $\Bbbk$ and line bundles $K,L$ on $X$.
\end{corollary}

\begin{corollary}\label{cor-3fold-7points}
Let $\Bbbk$ be a field of characteristic $0$.
For 3-dimensional smooth proper algebraic spaces $X$ over $\Bbbk$ and line bundles $K,L$ on $X$, Conjecture \ref{conj-1} holds modulo $Q^7$. Namely, the formulae resulted by expanding both sides of (\ref{eq-conj-1}) hold for Hilbert schemes of $\leq 6$ points.
\end{corollary}
\begin{proof}
By \cite[Proposition 5.8]{Hu21},  Conjecture \ref{conj-1} holds modulo $Q^7$ for toric 3-folds $X$ and equivariant line bundles on $X$. So the conclusion follows from Proposition \ref{cor-reduction-to-toric}.
\end{proof}

\begin{remark}\label{rmk-higher-ranks}
For vector bundles $E$, $F$ of any fixed ranks $r_1$ and $r_2$ respectively, there are some partial results in \cite{Wan16} and \cite{Kru18}.  But a general formula is still missing. In principle one can do the computations for the triples in $\mathcal{P}_{d,r_1,r_2}$ using the method of \cite[\S 3.4]{Hu21}.
\end{remark}

\begin{appendices}
\section{Moishezon's theorem for algebraic spaces}\label{sec:Moishezon-thm}
In this appendix we provide a proof of the Moishezon's theorem (\cite{Moi67}) for algebraic spaces. The results should be well-known, but we lack a reference. No originality is claimed.
\begin{theorem}[Moishezon]\label{thm-Moishezon-algsp}
Let $\Bbbk$ be a field of characteristic zero. Let $X$ be a smooth proper algebraic space over $\Bbbk$. Then there exists a proper morphism $\pi:\widetilde{X}\rightarrow X$, obtained by a finite number of successive blowing-ups along smooth centers, such that $\widetilde{X}$ is  projective over $\Bbbk$.
\end{theorem}
First we introduce a notion  of projective morphisms of algebraic spaces that behaves well with compositions and  blowing-ups, unconditionally.
\begin{definition}\label{def-CProj}
Let $S$ be a scheme. Let $f:Y\rightarrow X$ be a morphism of algebraic spaces over $S$. We say that $Y$ is a \emph{relative Proj of finite type over $X$}, if there exists a quasi-coherent sheaf $\mathcal{A}$ of graded $\mathcal{O}_X$-algebras of finite type, such that $Y$ is isomorphic to $\underline{\mathrm{Proj}}_{X}(\mathcal{A})$ as algebraic spaces over $X$. 
 We say that $f$ is a \emph{C-projective} morphism, if there exists a finite chain of morphisms of algebraic spaces over $S$
\[
Y=Y_n\xrightarrow{f_n}Y_{n-1}\xrightarrow{f_{n-1}}\dots\xrightarrow{f_{i+1}}Y_i\xrightarrow{f_i}\dots\xrightarrow{f_1}Y_0=X,
\]
such that $Y_i$ is a relative Proj of finite type over $Y_{i-1}$ via $f_i$, for $1\leq i\leq n$, and $f=f_n\circ\cdots\circ f_1$. Here the prefix C stands for “composition".
\end{definition}

\begin{lemma}\label{lem-Cprojmorphism}
Let $S$ be a scheme. The morphisms in the following statements are morphisms of algebraic spaces over $S$.
\begin{enumerate}
	\item[(i)] A closed immersion is C-projective;
	\item[(ii)] If $f:Y\rightarrow X$ is C-projective, and $g:Z\rightarrow X$ is a morphism, then  $f\times_X Z:Y\times_X Z\rightarrow Z$ is C-projective; 
	\item[(iii)] If $f:Z\rightarrow Y$ and $g:Y\rightarrow X$ are C-projective, then $g\circ f$ is C-projective;
	\item[(iv)] Let $\mathcal{I}$ be a quasi-coherent sheaf of finite type of ideals on $X$, and $Z$ be the closed subspace defined by $\mathcal{I}$. Then blowing up $X$ along $Z$ is  C-projective over $X$.
	\item[(v)] If $f:Z\rightarrow Y$ is a morphism, and $g:Y\rightarrow X$ is a separated morphism, such that $g\circ f$ is C-projective, then $f$ is C-projective.
\end{enumerate}
\end{lemma}
\begin{proof}
The statements (i), (iii) and (iv) are immediate from the definition.
The statement (ii) follows from successive applications of \cite[085C]{StPr}. Finally (v) follows from (i), (ii) and (iii) in a standard way (e.g. \cite[exercise II.4.8]{Har77}).
\end{proof}

We need also the elimination of indeterminacy for algebraic spaces. The following proof follows
\cite{RY02}.
\begin{lemma}[Elimination of indeterminacy]\label{lem-elim-indeterminacy}
Let $\Bbbk$ be a field of characteristic zero. Let $X$ and $Y$ be irreducible and reduced algebraic spaces proper over $\Bbbk$, and $f:X\dashrightarrow Y$ be a rational map. Then there is an algebraic space $X'$ and a morphism $g:X'\rightarrow X$, such that $g$ can be obtained by a finite number of successive blowing-ups along smooth centers, and $f\circ g$ extends to a morphism $X'\rightarrow Y$. 
\end{lemma}
\begin{proof}
By resolution of singularities (\cite{Hir64}) we can assume that $X$ is smooth.
Let $\Gamma\subset X\times_{\Bbbk} Y$ be the graph of $f$, and let $V$ be the closure of $\Gamma$ in $X\times_{\Bbbk} Y$ with a reduced algebraic space structure. By the Chow lemma \cite[IV.3.1]{Knu71}, there is a projective scheme $Z$ and a birational morphism $\phi:Z\rightarrow V$. Let $h:Z\rightarrow X$ be the composition $\mathrm{pr}_1\circ \phi$. Then there is a closed immersion $\iota:Z\rightarrow X\times \mathbb{P}^m$  over $X$. Denote by $\mathscr{L}$ the invertible sheaf $\iota^* \mathcal{O}(1)$, and let $\mathscr{S}=\mathcal{O}_X\oplus \bigoplus_{d\geq 1}h_* \mathscr{L}^d$, a graded sheaf of $\mathcal{O}_X$-modules of finite type, with $\mathscr{S}_d=h_* \mathscr{L}^d$. By the step 1 and 2 in \cite[proof of Theorem II.7.17]{Har77}, replacing $\iota$ by its $e$-tuple embedding if necessary for some $e\gg 1$, we have $Z\cong \underline{\mathrm{Proj}}_X \mathscr{S}$, and $\mathscr{S}$ is (étale) locally generated by $\mathscr{S}_1$. Since $\mathscr{L}$ is torsion free, $\mathscr{S}_1=h_* \mathscr{L}$ is torsion free. Since $h$ is birational, 
$\mathscr{S}_1$ is of rank one. As $X$ is smooth, the dual $\mathscr{S}_1^{\vee}$ is an invertible sheaf. The homomorphism $\mathscr{S}_1\rightarrow \mathscr{S}_1^{\vee\vee}$ is an injection, so $\mathscr{S}_1\otimes \mathscr{S}_1^{\vee}$ is a sheaf of ideals, and we denote it by $\mathscr{I}$. Then the step 5 of \cite[proof of Theorem II.7.17]{Har77} shows that $Z$ is isomorphic to the blowup of $X$ along $\mathscr{I}$. Using the functorial resolution of singularities of schemes \cite[Theorem 1.10]{BM97}, by a finite number of successive blowing-ups along smooth centers, we obtain $g:X'\rightarrow X$, such that $g^* \mathscr{I}$ is invertible. By the universal property of blowups \cite[085U]{StPr}, $g$ factors through $h$. Hence the proof is complete.
\end{proof}

\begin{proof}[Proof of Theorem \ref{thm-Moishezon-algsp}]
For a proof over $\mathbb{C}$ see e.g. \cite[Theorem 2.2.16]{MM07}. Our proof is a slightly modified one.
Without loss of generality we assume that $X$ is irreducible.
As an  algebraic space, $X$ has a dense open subset $U$ which is a scheme. By Nagata's compactification theorem (\cite{Nag63}) and Chow lemma, there is an irreducible projective scheme $Y$ and a birational map $f:Y\dashrightarrow X$. 
\begin{equation}\label{eq-graph-Moishezon}
	\xymatrix{\ar @{} [dr] |{\Box}
	Y_3 \ar[r]^{g_3} \ar[d]_{f_3} & Y_2 \ar[r]^{g_2} \ar[d]^{f_2} & Y_1 \ar[r]^{g_1} \ar[d]^{f_1} & Y \ar@{-->}[d]^{f} \\
	X_3 \ar[r]^{h_3} \ar@/_1pc/[rrr]_{h} & X_2 \ar[r]^{h_2} & X_1 \ar[r]^{h_1} & X
	}
\end{equation}
Since $\mathrm{char}(\Bbbk)=0$, using Lemma \ref{lem-elim-indeterminacy} we can find a commutative square, the rightmost one in (\ref{eq-graph-Moishezon}), where $g_1$ and $h_1$ can be obtained by finitely many times of blowups. Then by the flattening theorem \cite[théorèm 5.7.9]{RG71}, there is a blowup $h_2:X_2\rightarrow X_1$, such that the strict transform of $Y_1$, denoted by $Y_2$, is flat over $X_2$. In particular $g_2$ is itself a blowup of $Y_1$ and is thus a projective morphism (in the EGA sense). Now apply Lemma \ref{lem-elim-indeterminacy} to the rational map $(h_1\circ h_2)^{-1}:X\dashrightarrow X_2$, there exists $h:X_3\rightarrow X$ obtained by finitely many times of blowups along smooth centers, and a morphism $h_3:X_3\rightarrow X_2$, such that $h=h_1\circ h_2\circ h_3$. Since $h$ is C-projective, by Lemma \ref{lem-Cprojmorphism} (v) $h_3$ is also C-projective. Let $Y_3=X_3\times_{X_2}Y_2$. So $f_3$ is flat, and by Lemma \ref{lem-Cprojmorphism} (ii) $g_3$ is C-projective. But $Y_2$ is quasi-compact and quasi-separated scheme, consequently $g_3$ is projective in the EGA sense. Since $f_3$ is flat, birational and proper, and $X_3$ is smooth, by Zariski's main theorem $f_3$ is an isomorphism (note that $f_3$ is representable by schemes, by étale descent we can use Zariski's main theorem for schemes). The composition $g=g_1\circ g_2\circ g_3$ is a projective morphism, for each $g_i$ is. Hence $X_3$ is a smooth projective scheme, as  we want.
\end{proof}

\section{An identity of Chern classes in double point degenerations}\label{sec:identity-chern-doublepoint}
Let $\Bbbk$ be a field. Let $X\overset{\mathfrak{X}}{\rightsquigarrow}Y_1\cup_D Y_2$ be a double point degeneration, with the total space $\pi:\mathfrak{X}\rightarrow \mathbb{P}^1$, where $\mathfrak{X}$ is a smooth proper algebraic space over $\Bbbk$, and $\pi$ is a proper surjective morphism.  Suppose $\dim X=n$.  

Denote $i_X: X\hookrightarrow \mathfrak{X}$, $i_1:Y_1\hookrightarrow \mathfrak{X}$,  $i_2:Y_2\hookrightarrow \mathfrak{X}$, $j_1:D\hookrightarrow Y_1$, $j_2:D\hookrightarrow Y_2$, and also $i_D:D\hookrightarrow \mathfrak{X}=i_1\circ j_1=i_2\circ j_2$, to be the obvious imbeddings. Denote by $\pi_D:\mathbb{P}_D\rightarrow D$ the projection. Recall the \emph{sheaf of regular differentials}  $\widehat{\Omega}^1_{\mathfrak{X}/\mathbb{P}^1}=(\Omega^{1}_{\mathfrak{X}/\mathbb{P}^1})^{\vee\vee}$.
\begin{proposition}\label{prop-regulardiff}
\begin{enumerate}
\item[(1)] The sheaf of regular differentials $\widehat{\Omega}^1_{\mathfrak{X}/\mathbb{P}^1}$ is locally free of rank $n$.
\item[(2)] There are exact sequences
\begin{equation}\label{eq-regulardiff-exactseq-1}
\xymatrix{
  0 \ar[r] & \Omega^1_{Y_1/\Bbbk} \ar[r] & i_1^*\widehat{\Omega}^1_{\mathfrak{X}/\mathbb{P}^1} \ar[r] &
    j_{1*}\mathcal{O}_D \ar[r] & 0,\\
    }
\end{equation} 
and
\begin{equation}\label{eq-regulardiff-exactseq-2}
\xymatrix{   
    0 \ar[r] & \Omega^1_{Y_2/\Bbbk} \ar[r] & i_2^*\widehat{\Omega}^1_{\mathfrak{X}/\mathbb{P}^1} \ar[r] &
    j_{2*}\mathcal{O}_D \ar[r] & 0.
}
\end{equation}
\item[(3)] $N_{Y_1/D}\cong j_1^* \mathcal{O}_{Y_1}(D)$, $N_{Y_2/D}\cong j_2^* \mathcal{O}_{Y_2}(D)$, and $j_1^* \mathcal{O}_{Y_1}(D)\otimes j_2^* \mathcal{O}_{Y_2}(D)\cong \mathcal{O}_D$.
\end{enumerate}
\end{proposition}
\proof Near an ordinary double point $y$ on a fiber $\pi^{-1}(x)$, taking a suitable \'{e}tale neighborhood of $y$ on $\mathfrak{X}$, in local coordinates $\pi$ can be written as 
$\Bbbk[s]\rightarrow \Bbbk[y_1,\cdots,y_{n+1}]$ with $s\mapsto y_1 y_2$. By a computation in this local model, $(\Omega_{\mathfrak{X}/\mathbb{P}^1}^1)^{\vee}$ is locally free, and $(\Omega_{\mathfrak{X}/\mathbb{P}^1}^1)^{\vee\vee}$ is generated by $\Omega_{\mathfrak{X}/\mathbb{P}^1}^1$ and the meromorphic differential $\frac{dy_1}{y_1}=-\frac{dy_2}{y_2}$. Then the map $i_b^*\widehat{\Omega}^1_{\mathfrak{X}/\mathbb{P}^1}\rightarrow j_b \mathcal{O}_D$, for $b=1,2$, is the residue map, and the kernel is identified with $\Omega^1_{Y_b/\Bbbk}$. So we obtain (1) and (2).
 For (3), we use the adjunction formula, and the facts $\mathcal{O}_{Y_1}(D)=\mathcal{O}_{Y}(Y_2)|_{Y_1}$, $\mathcal{O}_{Y_2}(D)=\mathcal{O}_{Y}(Y_1)|_{Y_2}$, and $\mathcal{O}_{Y}(Y_1+Y_2)|_D\cong \mathcal{O}_D$. \pqed

By e.g. \cite{Vis89} and \cite{Kre99}, the intersection theory on algebraic space is well-defined and has the usual properties as for schemes.

\begin{theorem}\label{thm-chern-identity-doublepointdegeneration}
Let  $n_1,\cdots,n_m\geq 0$ be integers. Then in $\mathrm{CH}_*(\mathfrak{X})$ we have
\begin{multline}\label{eq-chern-identity-doublepointdegeneration}
i_{X*}\big(c_{n_1}(\Omega^1_{X/\Bbbk})\cdots c_{n_m}(\Omega^1_{X/\Bbbk})\cap[X]\big)=i_{1*}\big(c_{n_1}(\Omega^1_{Y_1/\Bbbk})\cdots c_{n_m}(\Omega^1_{Y_1/\Bbbk})\cap[Y_1])\\
+i_{2*}\big(c_{n_1}(\Omega^1_{Y_2/\Bbbk})\cdots c_{n_m}(\Omega^1_{Y_2/\Bbbk})\cap[Y_2]\big)
-i_{D*}\Big(\pi_{D*}\big(c_{n_1}(\Omega^1_{\mathbb{P}_D/\Bbbk})\cdots c_{n_m}(\Omega^1_{\mathbb{P}_D/\Bbbk})\cap[\mathbb{P}_D]\big)\Big).
\end{multline}
\end{theorem}
\begin{proof}  For a vector bundle $E$ and an indeterminate $t$, let $c(E)$ denote the total Chern class, and  $c_t(E)=\sum_{i=0}^{\infty}t^i c_i(E)$. Since $[X]=[\pi^{-1}(0)]$ in $\mathrm{CH}(\mathfrak{X})$, the theorem follows from the following  two equalities in $\mathrm{CH}(\mathfrak{X})[t_1,\dots,t_m]$:
\begin{equation}\label{eq-proof-thm-chern-identity-doublepointdegeneration-(-1)}
	c_{t_1}(\widehat{\Omega}^1_{\mathfrak{X}/\mathbb{P}^1})\cdots c_{t_m}(\widehat{\Omega}^1_{\mathfrak{X}/\mathbb{P}^1})\cap[X]=c_{t_1}(\Omega^1_{X/\Bbbk})\cdots c_{t_m}(\Omega^1_{X/\Bbbk})\cap[X]
\end{equation}
and
\begin{multline}\label{eq-proof-thm-chern-identity-doublepointdegeneration-0}
c_{t_1}(\widehat{\Omega}^1_{\mathfrak{X}/\mathbb{P}^1})\cdots c_{t_m}(\widehat{\Omega}^1_{\mathfrak{X}/\mathbb{P}^1})\cap[\pi^{-1}(0)]=c_{t_1}(\Omega^1_{Y_1/\Bbbk})\cdots c_{t_m}(\Omega^1_{Y_1/\Bbbk})\cap[Y_1]\\
+c_{t_1}(\Omega^1_{Y_2/\Bbbk})\cdots c_{t_m}(\Omega^1_{Y_2/\Bbbk})\cap[Y_2]
-\pi_{D*}\big(c_{t_1}(\Omega^1_{\mathbb{P}_D/\Bbbk})\cdots c_{t_m}(\Omega^1_{\mathbb{P}_D/\Bbbk})\cap[\mathbb{P}_D]\big).
\end{multline}
Here and in the following, the closed immersion pushforwards (of cycles) are omitted when there is no chance of confusion.  
By $i_{X}^* \widehat{\Omega}^1_{\mathfrak{X}/\mathbb{P}^1}=\Omega^1_{X/\Bbbk}$ and the projection formula we get (\ref{eq-proof-thm-chern-identity-doublepointdegeneration-(-1)}). We are left to show (\ref{eq-proof-thm-chern-identity-doublepointdegeneration-0}).
By Proposition \ref{prop-regulardiff}, 
\begin{eqnarray*}
c(i_b^*\widehat{\Omega}^1_{\mathfrak{X}/\mathbb{P}^1})=c(\Omega^1_{Y_b/\Bbbk})c(j_{b*}\mathcal{O}_D)=\frac{c(\Omega^1_{Y_b/\Bbbk})c(\mathcal{O}_{Y_b})}{c(\mathcal{O}_{Y_b}(-D))}=\frac{c(\Omega^1_{Y_b/\Bbbk})}{c(\mathcal{O}_{Y_b}(-D))}
\end{eqnarray*}
for $b=1,2$. Let $h_l(t_1,\cdots,t_m)$ denote the completely symmetric polynomial of degree $l$ of $t_1,\dots,t_m$, namely,
$$
\sum_{l=0}^{\infty}h_l(t_1,\cdots,t_m)x^l=\prod_{a=1}^{m}\frac{1}{1-t_a x}.
$$
Thus
\begin{eqnarray}\label{eq-proof-thm-chern-identity-doublepointdegeneration-1}
&& c_{t_1}(\widehat{\Omega}^1_{\mathfrak{X}/\mathbb{P}^1})\cdots c_{t_m}(\widehat{\Omega}^1_{\mathfrak{X}/\mathbb{P}^1})\cap[Y_1]\notag\\
&=& \prod_{a=1}^{m}\Big(c_{t_a}(\Omega^1_{Y_1/\Bbbk})\cdot\sum_{l=0}^{\infty}t_a^l c_1(\mathcal{O}_{Y_1}(D))^l\Big)\cap [Y_1]\notag\\
&=& \Big(\prod_{a=1}^{m}c_{t_a}(\Omega^1_{Y_1/\Bbbk})\cdot\sum_{l=0}^{\infty}h_l(t_1,\cdots,t_m) c_1(\mathcal{O}_{Y_1}(D))^l\Big)\cap [Y_1]\notag\\
&=& \prod_{a=1}^{m}c_{t_a}(\Omega^1_{Y_1/\Bbbk})\cap [Y_1]+
\prod_{a=1}^{m}\Big(c_{t_a}(\Omega^1_{Y_1/\Bbbk})\cdot
\sum_{l=1}^{\infty}h_l(t_1,\cdots,t_m) c_1(\mathcal{O}_{Y_1}(D))^{l-1}\Big)\cap[D].
\end{eqnarray}
Let $e_l(t_1,\cdots,t_m)$ be the elementary symmetric polynomial of $t_1,\dots,t_m$ of degree $l$. Then
\begin{equation}\label{eq-proof-thm-chern-identity-doublepointdegeneration-2}
	\prod_{a=1}^{m}(1-t_a x)\cdot \frac{1}{x}\Big(\prod_{a=1}^{m}\frac{1}{1-t_a x}-1\Big)
=\frac{1-\prod_{a=1}^{m}(1-t_a x)}{x}
=\sum_{l=1}^{m}(-1)^{l-1}e_{l}(t_1,\cdots,t_m)x^{l-1}.
\end{equation}
By the exact sequence
$$
\xymatrix{
  0 \ar[r] & N_{Y_1/D}^{\vee} \ar[r] & j_1^* \Omega^1_{Y_1/\Bbbk}\ar[r] & \Omega^1_{D/\Bbbk} \ar[r] & 0
}
$$
we have
\begin{equation}\label{eq-proof-thm-chern-identity-doublepointdegeneration-3}
c_t(j_1^* \Omega^1_{Y_1/\Bbbk})=c_t(\Omega^1_{D/\Bbbk})(1-t c_1(\mathcal{O}_{Y_1}(D))).
\end{equation}
From  (\ref{eq-proof-thm-chern-identity-doublepointdegeneration-2}) and (\ref{eq-proof-thm-chern-identity-doublepointdegeneration-3})  it follows that
\begin{eqnarray}\label{eq-proof-thm-chern-identity-doublepointdegeneration-4}
&&\prod_{a=1}^{m}\big(c_{t_a}(\Omega^1_{Y_1/\Bbbk})\cdot
\sum_{l=1}^{\infty}h_l(t_1,\cdots,t_m) c_1(\mathcal{O}_{Y_1}(D))^{l-1}\big)\cap[D]\notag\\
&=& \prod_{a=1}^{m}\big(c_{t_a}(\Omega^1_{D/\Bbbk})\cdot
\sum_{l=1}^{m}(-1)^{l-1}e_{l}(t_1,\cdots,t_m) c_1(\mathcal{O}_{Y_1}(D))^{l-1}\big)\cap[D].
\end{eqnarray}
By Proposition \ref{prop-regulardiff} (3),
\begin{equation}\label{eq-proof-thm-chern-identity-doublepointdegeneration-5}
j_1^* c_1(\mathcal{O}_{Y_1}(D))+j_2^* c_1(\mathcal{O}_{Y_2}(D))=0.
\end{equation}
So
\begin{eqnarray}\label{eq-proof-thm-chern-identity-doublepointdegeneration-6}
&& c_{t_1}(\widehat{\Omega}^1_{\mathfrak{X}/\mathbb{P}^1})\cdots c_{t_m}(\widehat{\Omega}^1_{\mathfrak{X}/\mathbb{P}^1})\cap[\pi^{-1}(0)]\notag\\
&=& \prod_{a=1}^{m}\Big(c_{t_a}(\Omega^1_{Y_1/\Bbbk})\cdot\sum_{l=0}^{\infty}t_a^l c_1(\mathcal{O}_{Y_1}(D))^l\Big)\cap [Y_1]\notag\\
&&+\prod_{a=1}^{m}\Big(c_{t_a}(\Omega^1_{Y_1/\Bbbk})\cdot\sum_{l=0}^{\infty}t_a^l c_1(\mathcal{O}_{Y_2}(D))^l\Big)\cap [Y_2]\notag\\
&\stackrel{\mbox{by (\ref{eq-proof-thm-chern-identity-doublepointdegeneration-1}) and (\ref{eq-proof-thm-chern-identity-doublepointdegeneration-4})}}{=}& \prod_{a=1}^{m}c_{t_a}(\Omega^1_{Y_1/\Bbbk})\cap [Y_1]+\prod_{a=1}^{m}c_{t_a}(\Omega^1_{Y_2/\Bbbk})\cap [Y_2]\notag\\
&&+\prod_{a=1}^{m}\big(c_{t_a}(\Omega^1_{D/\Bbbk})\cdot
\sum_{l=1}^{m}(-1)^{l-1}e_{l}(t_1,\cdots,t_m) c_1(\mathcal{O}_{Y_1}(D))^{l-1}\big)\cap[D]\notag\\
&&+\prod_{a=1}^{m}\big(c_{t_a}(\Omega^1_{D/\Bbbk})\cdot
\sum_{l=1}^{m}(-1)^{l-1}e_{l}(t_1,\cdots,t_m) c_1(\mathcal{O}_{Y_2}(D))^{l-1}\big)\cap[D]\notag\\
&\stackrel{\mbox{by (\ref{eq-proof-thm-chern-identity-doublepointdegeneration-5})}}{=}& \prod_{a=1}^{m}c_{t_a}(\Omega^1_{Y_1/\Bbbk})\cap [Y_1]+\prod_{a=1}^{m}c_{t_a}(\Omega^1_{Y_2/\Bbbk})\cap [Y_2]\notag\\
&&+2\prod_{a=1}^{m}\big(c_{t_a}(\Omega^1_{D/\Bbbk})\cdot
\sum_{0\leq l\leq \frac{m-1}{2}}^{m}e_{2l+1}(t_1,\cdots,t_m) c_1(\mathcal{O}_{Y_1}(D))^{2l}\big)\cap[D].
\end{eqnarray}

On the other hand, consider the exact sequences
\begin{equation}\label{eq-proof-thm-chern-identity-doublepointdegeneration-7}
	\xymatrix{
  0 \ar[r] & \pi_D^*\Omega^1_{D/\Bbbk} \ar[r] & \Omega^1_{\mathbb{P}_D/\Bbbk} \ar[r] & \Omega^1_{\mathbb{P}_D/D}  \ar[r] & 0,
  }
\end{equation}
and
\begin{equation}\label{eq-proof-thm-chern-identity-doublepointdegeneration-8}
 	\xymatrix{  
    0 \ar[r] & \Omega^1_{\mathbb{P}_D/D}  \ar[r] & \pi_D^*(N_{Y_1/D}^{\vee}\oplus \mathcal{O}_D)\otimes \mathcal{O}_{\mathbb{P}_D}(-1) \ar[r] & \mathcal{O}_{\mathbb{P}_D} \ar[r] & 0.
}
 \end{equation} 
From (\ref{eq-proof-thm-chern-identity-doublepointdegeneration-7}), one sees that 
$c_2\big(\pi_D^*(N_{Y_1/D}^{\vee}\oplus \mathcal{O}_D)\otimes \mathcal{O}_{\mathbb{P}_D}(-1))=0$, thus
\begin{equation}\label{eq-proof-thm-chern-identity-doublepointdegeneration-9}
	c_1(\mathcal{O}_{\mathbb{P}_D}(1))^2+\pi_D^* c_1(\mathcal{O}_{Y_1}(D))\cdot c_1(\mathcal{O}_{\mathbb{P}_D}(1))=0.
\end{equation}
So (\ref{eq-proof-thm-chern-identity-doublepointdegeneration-7})-(\ref{eq-proof-thm-chern-identity-doublepointdegeneration-9}) implies
\begin{eqnarray}\label{eq-proof-thm-chern-identity-doublepointdegeneration-10}
&&c(\Omega^1_{\mathbb{P}_D/\Bbbk})=\pi_D^* c(\Omega^1_{D/\Bbbk})\cdot c(\Omega^1_{\mathbb{P}_D/D})
=\pi_D^* c(\Omega^1_{D/\Bbbk})\cdot c\big(\pi_D^*(N_{Y_1/D}^{\vee}\oplus \mathcal{O}_D)\otimes \mathcal{O}_{\mathbb{P}_D}(-1)\big)\notag\\
&=& \pi_D^* c(\Omega^1_{D/\Bbbk})\cdot \big(1-\pi_D^* c_1(\mathcal{O}_{Y_1}(D))-c_1(\mathcal{O}_{\mathbb{P}_D}(1))\big)\big(1-c_1(\mathcal{O}_{\mathbb{P}_D}(1))\big)\notag\\
&=& \pi_D^* c(\Omega^1_{D/\Bbbk})\cdot \big(1-\pi_D^* c_1(\mathcal{O}_{Y_1}(D))-2c_1(\mathcal{O}_{\mathbb{P}_D}(1))\big).
\end{eqnarray}
Moreover, by (\ref{eq-proof-thm-chern-identity-doublepointdegeneration-9}), we have
\begin{eqnarray*}
&&\big(\pi_D^* c_1(\mathcal{O}_{Y_1}(D))+2c_1(\mathcal{O}_{\mathbb{P}_D}(1))\big)^2\\
&=&\pi_D^* c_1(\mathcal{O}_{Y_1}(D))^2+4\pi_D^* c_1(\mathcal{O}_{Y_1}(D))\cdot c_1(\mathcal{O}_{\mathbb{P}_D}(1)) +4c_1(\mathcal{O}_{\mathbb{P}_D}(1))^2\\
&=& \pi_D^* c_1(\mathcal{O}_{Y_1}(D))^2,
\end{eqnarray*}
and thus
\begin{eqnarray}\label{eq-proof-thm-chern-identity-doublepointdegeneration-11}
\big(\pi_D^* c_1(\mathcal{O}_{Y_1}(D))+2c_1(\mathcal{O}_{\mathbb{P}_D}(1))\big)^l=
\begin{cases}
 \pi_D^* c_1(\mathcal{O}_{Y_1}(D))^{l},& \mbox{if}\ 2|l\\
 \pi_D^* c_1(\mathcal{O}_{Y_1}(D))^{l-1}\cdot \big(\pi_D^* c_1(\mathcal{O}_{Y_1}(D))+2c_1(\mathcal{O}_{\mathbb{P}_D}(1))\big),& \mbox{if}\ 2\nmid l.
\end{cases}
\end{eqnarray}
By (\ref{eq-proof-thm-chern-identity-doublepointdegeneration-10}),
\begin{eqnarray*}
&&c_{t_1}(\Omega^1_{\mathbb{P}_D/\Bbbk})\cdots c_{t_m}(\Omega^1_{\mathbb{P}_D/\Bbbk})\cap[\mathbb{P}_D]\\
&=&\prod_{i=1}^{m}\pi_D^* c_{t_i}(\Omega^1_{D/\Bbbk})\cdot \prod_{i=1}^m\big(1-t_i\pi_D^* c_1(\mathcal{O}_{Y_1}(D))-2t_i c_1(\mathcal{O}_{\mathbb{P}_D}(1))\big)\cap[\mathbb{P}_D]\\
&=& \prod_{i=1}^{m}\pi_D^* c_{t_i}(\Omega^1_{D/\Bbbk})\cdot \sum_{l=0}^m(-1)^l e_l(t_1,\cdots, t_m)
\big(\pi_D^* c_1(\mathcal{O}_{Y_1}(D))+2c_1(\mathcal{O}_{\mathbb{P}_D}(1))\big)^l\cap[\mathbb{P}_D].
\end{eqnarray*}
Therefore by (\ref{eq-proof-thm-chern-identity-doublepointdegeneration-11}) and $c_1((\mathcal{O}_{\mathbb{P}_D}(1))\cap[\mathbb{P}_D]=[D]$ (see e.g. \cite[Proposition 3.1(a)]{Ful98}), and the projection formula, we get
\begin{eqnarray}\label{eq-proof-thm-chern-identity-doublepointdegeneration-12}
&&\pi_{D*}\big(c_{t_1}(\Omega^1_{\mathbb{P}_D/\Bbbk})\cdots c_{t_m}(\Omega^1_{\mathbb{P}_D/\Bbbk})\cap[\mathbb{P}_D]\big)\notag\\
&=&-2\big(\prod_{i=1}^{m} c_{t_i}(\Omega^1_{D/\Bbbk})\cdot \sum_{0\leq l\leq \frac{m-1}{2}} e_{2l+1}(t_1,\cdots, t_m)c_1(\mathcal{O}_{Y_1}(D))^{2l}\big)\cap[D].
\end{eqnarray}
Combining (\ref{eq-proof-thm-chern-identity-doublepointdegeneration-6}) and (\ref{eq-proof-thm-chern-identity-doublepointdegeneration-12}) we  obtain 
(\ref{eq-proof-thm-chern-identity-doublepointdegeneration-0}).
\end{proof}

\end{appendices}

\textsc{School of Sciences, Great Bay University, Dongguan, P.R. China}

 \emph{E-mail address:}  huxw06@gmail.com

\end{document}